\documentclass[a4paper,reqno]{amsart}
\usepackage{amssymb, amsmath, amscd}
\def\BB{\mathcal{B}}
\def\qedbox{\hbox{$\rlap{$\sqcap$}\sqcup$}}
\newtheorem{theorem}{Theorem}[section]

\newtheorem{lemma}{Lemma}[section]
\newtheorem{remark}{Remark}[section]
\newtheorem{constraint}{Constraint}[section]
\newtheorem{example}{Example}[section]
\newcommand{\R}{\mathbb{R}}
\makeatletter
 \@addtoreset{equation}{section}
\makeatother
\begin{document}
\title{Growth of heat trace and heat content asymptotic coefficients}
\author{M. van den Berg, Peter  Gilkey, and K. Kirsten}
\address{MvdB: School of Mathematics, University of Bristol,
University Walk, Bristol BS8 1TW, UK\smallbreak E-mail:
M.vandenBerg@bris.ac.uk}
\address{PG: Mathematics Department, University of Oregon,
  Eugene OR 97403\smallbreak
  E-mail: gilkey@uoregon.edu}
\address{KK: Department of Mathematics, Baylor University\\
Waco, Texas, TX 76798, USA\smallbreak
  E-mail: Klaus\_Kirsten@baylor.edu}
\begin{abstract}{We show in the smooth category that the heat trace asymptotics
and the heat content asymptotics can be made to grow arbitrarily
rapidly. In the real analytic context, however, this is not true
and we establish universal bounds on their growth.
\\MSC 2002: 35K20,35P99,58J25,58J50}\end{abstract}
\maketitle

\section{Introduction}\label{sect-1}

\subsection{Heat trace asymptotics}
Let $(M,g)$ be a compact Riemannian manifold of dimension $m$ with
smooth (possibly empty) boundary $\partial M$. Let
$\operatorname{dvol}_m$ and $\operatorname{dvol}_{m-1}$ be the
Riemannian volume elements on $M$ and on $\partial M$,
respectively. Let $\Delta_g$ be the scalar Laplacian. Let $\nu$ be the
inward unit normal on the boundary; we extend $\nu$ by parallel
translation to a vector field defined on a collared neighborhood
of the boundary so $\nabla_\nu\nu=0$; this means that the integral curves of
$\nu$ are unit speed geodesics perpendicular to $\partial M$. Let
$$\BB^-\phi:=\phi|_{\partial M}\quad\text{and}\quad\BB^+\phi:=\nu\phi|_{\partial M}$$
be the Dirichlet and Neumann boundary operators, respectively.  Impose boundary conditions $\BB=\BB^-$ or $\BB=\BB^+$.
Let $u:M\times(0,\infty)\rightarrow \R$ be the unique solution of
$$
\begin{array}{ll}
(\partial_t+\Delta_g)u(x,t)=0&\text{(heat equation)},\\
\lim_{t\rightarrow0}u(\cdot,t)=\phi_1(\cdot)\text{ in }L^2&\text{(initial condition)},\vphantom{\vrule height 11pt}\\
\BB u(\cdot,t)=0\text{ for }t>0&\text{(boundary condition)},\vphantom{\vrule height 11pt}
\end{array}
$$
where $\phi_1$ is real-valued and smooth on $M$. Then $u(x,t)$
represents the temperature at $x\in M$ at time $t>0$ if $M$ has
initial temperature distribution $\phi_1$ where the boundary condition $\BB$ is imposed on $u$ for $t>0$. The solution is
formally given by $${u(x,t):=e^{-t\Delta_{g,\BB}}\phi_1(x),}$$ 
where $\Delta_{g,\BB}$ is the associated realization of the Laplacian. The operator
$e^{-t\Delta_{g,\BB}}$ is a smoothing operator of trace class and, as
$t\downarrow0$, there is a complete asymptotic series of the form
\cite{Gri66,Gri71,Se66,Se68b,Se69a,Se69,Se69b}
$$
\operatorname{Tr}_{L^2}\{e^{-t\Delta_{g,\BB}}\}\sim(4\pi t)^{-m/2}\sum_{n=0}^\infty a_n(M,g,\BB)t^{n/2}\,.
$$
If $M$ is a closed manifold, the boundary condition plays no role
and we shall denote these coefficients by $a_n(M,g)$. They vanish
if $n$ is odd in this instance.

 The asymptotic coefficients $\{a_1,a_2,\cdots\}$ are locally computable invariants of $M$ and of
$\partial M$ as we shall see presently in Section~\ref{sect-2}. In mathematical physics, they occur for example in
the calculation of Casimir forces \cite{x2,BKMM09,x1} or in the study of
the partition function of quantum mechanical systems
\cite{x3,x4,x1}. They are known in the category of manifolds with
boundary for $n\le 5$ \cite{BGK99,Ki98}, and in the category of
closed manifolds for $n\le8$ \cite{AmBeOc89,Av91}. These coefficients
play a crucial role in the study of isospectral questions. Related
invariants for more general operators of Laplace type also play a
crucial role in the local index theorem. See, for example, the
discussion and references in \cite{A89, A90, BGO90, BPY89, CY89,
G88, G94, PG, McSi67, MinPle49, OPS88}. They have also been
studied with nonlocal boundary conditions \cite{K08}. We also
refer to \cite{GW10} where the heat trace itself is studied and
not just the asymptotic coefficients. For the study of the
asymptotic behaviour of the eigenvalues of $\Delta_{g,\BB}$ we
refer to \cite{SafVas97} and the references therein. The field is
vast and it is only possible to cite a few references.

\subsection{Planar domains}\label{sect-1.2} In the case of a planar domain $\Omega$, the heat trace asymptotics (with
Dirichlet boundary conditions) have been
computed for $n\le13$ by Berry and Howls \cite{x5}. Berry and
Howls computed $a_n$ for $n\le 31$ in the case of a disc
\cite{x5}, and were led to conjecture that for planar domains
$\Omega$ and for $n\rightarrow\infty$
\begin{equation}\label{eqn-1.a}
a_n(\Omega)=\alpha \Gamma(n-\beta+1)\Gamma(
n/2)^{-1}\ell(\Omega)^{2-n}(1+o(1)),
\end{equation}
where $\alpha$ and $\beta$ are dimensionless quantities and where
$\ell(\Omega)$ is the length of the shortest accessible periodic
geodesic in $\Omega$. In particular, for a disk of radius $R$ and
shortest accessible periodic geodesic $4R$, they further
conjectured that Equation (\ref{eqn-1.a}) holds with
$\alpha=(8\sqrt{2\pi})^{-1}$ and $\beta=\frac32$. While the latter
conjecture remains open to date, it is instructive to see that
Equation (\ref{eqn-1.a}) can not hold in general. The following
counter examples were given in \cite{x6}.
\begin{example}\rm Let $0<\varepsilon<\frac15$, and let
\medbreak\qquad $\tilde
P_\varepsilon=\{(x_1,x_2)\in\mathbb{R}^2:|x|\le1,|x_2|\le1-\varepsilon\}$,
\medbreak\qquad $\tilde
Q_\varepsilon=\{(x_1,x_2)\in\mathbb{R}^2:|x|\le1,x_1\le1-\varepsilon,x_2\le1-\varepsilon\}$.
\medbreak\noindent We smooth out the corners of $\partial\tilde
P_\varepsilon$ at $x_2=\pm(1-\varepsilon)$ and of $\partial\tilde
Q_\varepsilon$ at $x_1=1-\varepsilon$, $x_2=1-\varepsilon$
isometrically to obtain two convex domains $P_\varepsilon$ and
$Q_\varepsilon$ with smooth boundary and with
$a_n(P_\varepsilon)=a_n(Q_\varepsilon)$ and
$\ell(P_\varepsilon)=4(1-\varepsilon)$,
$\ell(Q_\varepsilon)=2(2-\varepsilon)$. This then contradicts
Equation (\ref{eqn-1.a}).
\end{example}

\begin{example}\rm Let $0<\varepsilon<1$, $0<\rho<1-\varepsilon$,
and let \medbreak\quad
$\Omega_{\varepsilon}:=\{(x_1,x_2)\in\mathbb{R}^2:\varepsilon\le
|x|\le1\}$, \medbreak\quad
$\Omega_{\varepsilon}^\rho:=\{(x_1,x_2)\in\mathbb{R}^2:|x|\le1$,
$|(x_1-\rho,x_2)|\ge\varepsilon\}$. \medbreak\noindent We then
have that
$a_n(\Omega_\varepsilon)=a_n(\Omega_\varepsilon^\varrho)$ and
$\ell(\Omega_\varepsilon)=2(1-\varepsilon)$,
$\ell(\Omega_\varepsilon^\rho)=2(1-\varepsilon-\rho)$ which once
again contradicts Equation (\ref{eqn-1.a}).
\end{example}

It remains an open problem to construct a pair of iso - $a_n$ real
analytic simply connected planar domains which have different
shortest periodic geodesics. It has been conjectured that Equation
(\ref{eqn-1.a}) also holds for balls in $\mathbb{R}^m$ where
$\beta$ depends upon $d$ only \cite{x7}.

\subsection{The heat trace asymptotics in the real analytic category} The calculus of Seeley
\cite{Se66,Se68b,Se69a,Se69,Se69b} and Greiner \cite{Gri66,Gri71}
shows that $a_n$ is given by a local formula; the following result will then follow from the
analysis of Section~\ref{sect-2}:

\goodbreak\begin{theorem}\label{thm-1.1}
Let $\BB$ be either Dirichlet or Neumann boundary conditions. There exist universal constants $\kappa_{n,m}$ so that if
$(M,g)$ is any compact real analytic manifold of dimension $m$, then there
exists a positive constant $C=C(M,g)$ such that
$$|a_n(M,g,\BB)|\le\kappa_{n,m} C^n\cdot\operatorname{vol}_m(M,g)\quad\text{for any}\quad n\,.$$
\end{theorem}

We note some similarity between the formulae of Equation
(\ref{eqn-1.a}) and Theorem~\ref{thm-1.1}. The geometric data of
$(M,g)$ appear in $C^n$, whereas the prefactor is of a
combinatorial nature and depends on $m$ and $n$ only.
We can choose the constant to rescale appropriately under homotheties, i.e. so that $C(M,c^2g)=c^{-1}C(M,g)$.

We restrict momentarily to the context of closed manifolds, i.e. compact manifolds with empty boundary. We adopt the
Einstein convention and sum over repeated indices. We say that
$D$ is an {\it operator of Laplace type}, if in any local system
of coordinates we may express $D$ in the form:
\begin{equation}\label{eqn-1.b}
D=-\left(g^{ij}\partial_{x_i}\partial_{x_j}+A^k\partial_{x_k}+B\right)\,.
\end{equation}
 Let $a_n(x,D)$ be the
local heat trace invariant of such an operator. We shall primarily interested in the case $n$ even so we
shall set $n=2\bar n$ in what follows.  If $f$ is
any smooth function on $M$, then
\begin{equation}\label{eqn-1.c}
\operatorname{Tr}_{L^2}(fe^{-tD})\sim(4\pi t)^{-m/2}\sum_{\bar n=0}^\infty t^{\bar n}
\int_M a_{2\bar n}(x,D)f(x)\operatorname{dvol}_m\,.
\end{equation}
The following result shows that the factorial growth
conjectured by Berry and Howls for planar domains pertains in this setting as well as regards the local heat trace
invariants on closed manifolds.
\goodbreak\begin{theorem}\label{thm-1.2}
Let $(M,g)$ be a closed real analytic Riemannian manifold of dimension $m\ge2$.
\begin{enumerate}
\item Let $D$ be a scalar real analytic operator of
Laplace type on
$M$. Then there exists a constant $C_1=C_1(M,g,D)$ so that 
$$|a_{2\bar n}(x,D)|\le C_1^{\bar n}\cdot\bar n!\quad\text{for any}\quad\bar n\ge1\,.$$
\item Let $P$ be a point of $M$. Suppose there
exists a real analytic function $f$ on
$M$ such that
$df(P)\ne0$. Then there exists a constant $C_2=C_2(P,M,g,f)>0$ and there exists a real analytic function $h$ on $M$
so that the conformally equivalent metric
$g_h:=e^{2h}g$ satisfies
$$|a_{2\bar n}(P,\Delta_{g_h})|\ge C_2^{\bar n}\cdot\bar n!\quad\text{for any}\quad\bar n\ge3\,.$$
\end{enumerate}
\end{theorem}

\begin{remark}\rm
Assertion (1) can be integrated to yield an upper bound on the heat trace asymptotics $a_{2\bar n}(D)$. However, Assertion
(2) is only valid at a single point of $M$. Since it in fact arises from considering a divergence term in the local
expansion, we do not obtain a corresponding estimate for $a_{2\bar n}(D)$.
\end{remark}

\subsection{The heat trace asymptotics in the smooth category} The situation in the smooth non real analytic setting is very
different. Fix a background reference Riemannian metric $h$ and let $\nabla^h$ be the associated Levi-Civita
connection which we use to covariantly differentiate tensors of all types. If $T$ is a tensor field on $M$, we
define the $C^k$ norm of $T$ by setting:
$$||T||_k:=\max_{P\in M}\left\{\sum_{i=0}^k|\nabla^{h,i}T|(P)\right\}\,.$$
Changing $h$ replaces $||T||_k$ by an equivalent norm; we therefore suppress the dependence upon $h$.
But as we will be changing the metric when considering the heat trace asymptotics subsequently, it is useful to
have fixed
$h$ once and for all so the associated $C^k$ norms do not change. Theorem~\ref{thm-1.1} fails in the smooth context as we
have:

\goodbreak\begin{theorem}\label{thm-1.3}
Let $k\ge3$ be given, let constants $C_{\bar n}>0$ for $\bar n\ge k$ be given,
and let $\epsilon>0$ be given. Let $(M,g)$ be a smooth compact Riemannian
manifold of dimension $m\ge2$ without boundary and let $g_e$ be
the usual Euclidean metric on $\mathbb{R}^{m+1}$.
\begin{enumerate}
\item There exists a function $f\in
C^\infty(M)$ with $||f||_{k-1}<\epsilon$ so that if $g_1:=e^{2f}g$ is the conformally related metric, then 
$$|a_{2\bar n}(M,g_1)|\ge
C_{\bar n}\quad\text{for any}\quad \bar n\ge k\,.$$
\item Suppose that $g=\Theta^*g_e$ where $\Theta$ is an immersion of $M$ into $\mathbb{R}^{m+1}$. There exists an
immersion
$\Theta_1$ with $||\Theta-\Theta_1||_{k-1}<\epsilon$ so that if ${g_1:=\Theta_1^*g_e}$, then
$$|a_{2\bar n}(M,g_1)|\ge C_{\bar n}\quad\text{for any}\quad \bar n\ge k\,.$$
\end{enumerate}
\end{theorem}

\subsection{Heat content asymptotics}
There are analogous results for the heat content asymptotics.  Let $\phi_1$ be the initial temperature of the
manifold and let $\phi_2$ be the specific heat of the manifold. We
suppose throughout that $\phi_1$ and $\phi_2$ are smooth. The
total heat energy content of the manifold is then given by:
$$\beta(\phi_1,\phi_2,\Delta_g,\BB)(t):=\int_Mu(x,t)\phi_2(x)\operatorname{dvol}_m\,.$$
As $t\downarrow0$, there is a complete asymptotic expansion of the form
\begin{eqnarray*}
\beta(\phi_1,\phi_2,\Delta_g,\BB)(t)&\sim&\sum_{n=0}^\infty
\frac{(-t)^n}{n!}\int_M\Delta_g^n\phi_1\cdot\phi_2\operatorname{dvol}_m\\
&+&
\sum_{\ell=0}^\infty t^{(\ell+1)/2}\beta_\ell^{\partial M}(\phi_1,\phi_2,\Delta_g,\BB)\,.
\end{eqnarray*}
The coefficients involving integrals over $M$ arise from the heat
redistribution on the interior of the manifold and are well
understood. The additional boundary terms $\beta_\ell^{\partial
M}$ are the focus of our inquiry. They, like the heat trace
asymptotics, are given by local formulae and have been studied
extensively (see, for example
\cite{BeDeGi93,BeGi94,BGGK10,BGKK07,BGS08,Ca45,McA93,McMe03a,PhJa90,Sav98a}
and the references contained therein).

Inspired by the work of Howls and Berry \cite{x7}, Trav\v{e}nec
and \v{S}amaj \cite{TS11} investigated the asymptotic behaviour of
the coefficients $\beta_{\ell}$ as $\ell \rightarrow \infty$ in
 flat space in the special case that $\phi_1=\phi_2=1$ with Dirichlet boundary conditions. The interior invariants
then play no role for $n\ge1$ and one has, adopting the notational conventions of
this paper, that
$$
\beta(1,1,\Delta_g,\mathcal{B}^-)(t)\sim\operatorname{vol}_m(M,g)+\sum_{\ell=0}^\infty
t^{(\ell+1)/2}\beta_\ell^{\partial M}(1,1,\Delta_g,\mathcal{B}^-)\,.
$$
After interpreting the results of \cite{TS11} in our notation, they found that if
$M$ is a ball in
$\mathbb{R}^m$ of radius
$r$ with $m$ {\bf even}, then as $\ell\rightarrow\infty$ one has:
\begin{equation}\label{eqn-1.d}
\beta_{\ell}=4\pi^{(m-3)/2}\Gamma(m/2)^{-1}(\ell+1)^{-1}\Gamma(\ell/2)r^{m-\ell-1}(1+o(1))\,.
\end{equation}

The structure of Equation (\ref{eqn-1.d}) is similar to that of
Equation (\ref{eqn-1.a}). There is a combinatorial coefficient
in $m$ and $\ell$, while the shortest periodic geodesic appears
to a suitable power. However, for $m$ {\bf odd} Trav\v{e}nec and
\v{S}amaj obtained polynomial dependence rather than factorial
dependence of $\beta_{\ell}^{\partial M}$ in $\ell$ \cite{TS11}. Furthermore
the two examples in Section~\ref{sect-1.2} above provide iso-$\beta_{\ell}$
pairs of smooth planar domains with different shortest periodic
geodesic lengths. Hence the structure of the asymptotic behaviour
of the $\beta_{\ell}$'s in flat space remains unclear in general.

For $\ell$ even, the boundary term involves a fractional power of
$t$ and there is no corresponding interior term. This simplifies
the control of these terms. Consequently, we shall usually set $\ell=2\bar\ell$ in what follows.

\subsection{The heat content asymptotics in the real analytic setting}
 As noted above, results of \cite{TS11} showed that the heat content asymptotics on
the ball in $\mathbb{R}^m$ for $m$ even exhibit
growth rates similar to that given in Theorem~\ref{thm-1.2} for the local heat trace asymptotics. We generalize
Theorem~\ref{thm-1.2} (2) to this setting to derive an estimate using conformal variations which shows that the metric on the
boundary does not play a central role in the analysis:

\goodbreak\begin{theorem}\label{thm-1.4}
Let $m\ge2$.\begin{enumerate}
\item Let $(N,g_N)$ be a closed Riemannian manifold of dimension $m-1$. Let $M:=[0,2\pi]\times N$.
There exists a real analytic function $h(x)$ on $[0,2\pi]$, which depends on the choice of
$(N,g_N)$, so that the conformally adjusted metric $g_M:=e^{2h}\{dx^2+g_N\}$
satisfies:
$$\left|\beta_{2\bar\ell}^{\partial M}(1,1,\Delta_{g_M},\BB^-)\right|\ge 
\bar\ell!\cdot\operatorname{vol}_{m-1}(N,g_N)\quad\text{for any}\quad\bar\ell\ge3\,.$$
\item Let $g_e$ be the standard Euclidean metric on the unit disk $D^m$ in $\mathbb{R}^m$.
There exists a radial real analytic function $h$ on $D^m$, which depends on $m$,  so that the conformally adjusted product
metric
$g_M:=e^{2h}g_e$ satisfies:
$$\left|\beta_{2\bar\ell}^{\partial M}(1,1,\Delta_{g_M},\BB^-)\right|\ge 
\bar\ell!\cdot\operatorname{vol}_{m-1}(N,g_N)\quad\text{for any}\quad\bar\ell\ge3\,.$$
\end{enumerate}
\end{theorem}

We have estimates for the heat content asymptotics in this setting which are similar to those given in Theorem~\ref{thm-1.1}:
\goodbreak\begin{theorem}\label{thm-1.5}
There exist universal constants $\kappa_{n,m}$ and $\tilde\kappa_{\ell,m}$ such that if
$(M,g)$ is a compact real analytic Riemannian manifold of
dimension $m$ and if $(\phi_1,\phi_2)$ are real analytic, then
there exists a positive constant $C=C(M,g,\phi_1,\phi_2,\BB)$ such
that
\begin{eqnarray*}
&&\left|\int_M\phi_1\cdot\Delta_g^n\phi_2\operatorname{dvol}_m\right|\le\kappa_{n,m}
C^n\cdot\operatorname{vol}_m(M,g),\\
&&\left|\beta_\ell^{\partial M}(\phi_1,\phi_2,\Delta_g,\BB^\pm)\right|\le\tilde\kappa_{\ell,m} C^\ell
\cdot\operatorname{vol}_{m-1}(\partial M,g)\,.
\end{eqnarray*}
\end{theorem}

\begin{remark}\rm Again, the constant $C$ can be chosen so that $$C(M,c^2g)=c^{-2}C(M,g)\,.$$
\end{remark}

\subsection{The heat content asymptotics in the smooth setting}
Theorem~\ref{thm-1.5} fails in the smooth setting as we have:

\goodbreak\begin{theorem}\label{thm-1.6}
Let $k\ge3$ be given, let constants $C_{\bar\ell}>0$ for $\bar\ell\ge k$ be given,
and let $\epsilon>0$ be given. 
Let $\BB=\BB^+$ or $\BB=\BB^-$. Let $(M,g)$ be a smooth compact Riemannian
manifold of dimension $m\ge1$ with non-trivial boundary. Let $\phi_1$ be a smooth initial temperature
 and let $\phi_2$ be a smooth specific heat with $\BB\phi_2\ne0$. There exists $\Phi_1$ with
$||\phi_1-\Phi_1||_{2k-1}<\varepsilon$ such that:
$$\beta_{2\bar\ell}^{\partial M}(\Phi_1,\phi_2,\Delta_g,\BB)= C_{\bar\ell}\quad\text{for any}\quad\bar\ell\ge k\,.$$
\end{theorem}

The heat content asymptotics were originally studied for Dirichlet
boundary conditions and for $\phi_1=\phi_2=1$
\cite{BD89,BG94,BS90}. We have the following theorem in this
setting:

\goodbreak\begin{theorem}\label{thm-1.7}
Let $k\ge3$ be given, let constants $C_{\bar\ell}>0$ for $\bar\ell\ge k$ be given,
and let $\epsilon>0$ be given. Let $(M,g)$ be a smooth compact
manifold Riemannian manifold of dimension $m\ge2$ with non-trivial boundary. There exists a metric $g_1$ so
$||g-g_1||_{2k-1}<\varepsilon$ such that 
$$\beta_{2\bar\ell}^{\partial M}(1,1,\Delta_{g_1},\BB^-)=C_{\bar\ell}\quad\text{for any}\quad\bar\ell\ge k\,.$$
\end{theorem}

\subsection{Bochner formalism for operators of Laplace type}\label{sect-1.8}
The results given above in Theorem~\ref{thm-1.3}, in Theorem~\ref{thm-1.6}, and in
Theorem~\ref{thm-1.7} rely upon a leading term analysis of the heat trace asymptotics and of the heat content asymptotics. It
is one of the paradoxes of this subject that to apply the functorial method, one must work with very general operators even
if one is only interested in the scalar Laplacian, as is the case in this
paper. We only consider the context of scalar operators. There is
a corresponding notion for systems, i.e. operators which act on
the space of smooth sections to some vector bundle. It is possible to express an operator $D$ of Laplace type as given
in Equation (\ref{eqn-1.b}) invariantly using a Bochner formalism \cite{G94}. There exists a unique connection $\nabla$
and a unique smooth function $E$ so that
$$D\phi=-(g^{uv}\phi_{;uv}+E\phi)\, ,$$
where we use `;' to denote the components of multiple covariant differentiation with respect to $\nabla$ and with respect to
the Levi-Civita connection. Let
$\Gamma_{uv}{}^w$ be the Christoffel symbols of the Levi-Civita connection and let $\omega$ be the connection $1$-form of
$\nabla$. We then have
\begin{equation}\label{eqn-1.e}
\begin{array}{l}
\omega_u=\textstyle\frac12g_{uv}(A^v+g^{sw}\Gamma_{sw}{}^v
\operatorname{Id}),\\
E=B-g^{uv}(\partial_{x_u}\omega_{v}+\omega_{u}\omega_{v}
    -\omega_{w}\Gamma_{uv}{}^w)\,.\vphantom{\vrule height 11pt}
\end{array}\end{equation}

\subsection{Leading term analysis}\label{sect-1.9}
 Theorem~\ref{thm-1.8} below will play a central role in our
analysis, and was established in \cite{BGO90,G79,G88}. We also refer
to related work in the $2$-dimensional setting \cite{OPS88}. It
has been used by Brooks, Perry, Yang \cite{BPY89} and by Chang and
Yang \cite{CY89} to show families of isospectral metrics within a
conformal class are compact modulo gauge equivalence in dimension
3. Let $\tau$ be the
scalar curvature of $g$, let $\rho$ be the Ricci tensor of $g$, and let $\Omega$ be the curvature
of the connection $\nabla$ defined by an operator of Laplace type.
\goodbreak\begin{theorem}\label{thm-1.8}
Let $D$ be an operator of Laplace type on a closed Riemannian manifold $(M,g)$ and let $\bar n\ge3$.
\begin{enumerate}
\item The local heat trace asymptotics satisfy:
\begin{eqnarray*}
a_{2\bar n}(P,\Delta_g)&=&\frac{(-1)^{\bar n}\bar n!}{(2\bar n+1)!}\{-\bar n
\Delta^{\bar n-1}\tau-(4n+2)\Delta^{\bar n-1}E\}\\&+&\text{lower
order derivative terms}\,.
\end{eqnarray*}
\item The global heat trace asymptotics satisfy:
\begin{eqnarray*}
a_{2\bar n}(D)&=&\frac12\frac{(-1)^{\bar n}\bar n!}{(2\bar n+1)!}\int_M\{(\bar n^2-\bar
n-1)|\nabla^{\bar n-2}\tau|^2
+2|\nabla^{\bar n-2}\rho|^2\\
&+&4(2\bar n+1)(\bar n-1)\nabla^{(\bar n-2)}\tau\cdot\nabla^{(\bar
n-2)}E+2(2\bar n+1)|\nabla^{(\bar n-2)}\Omega|^2\\
&+&4(2\bar n-1)(2\bar n+1)|\nabla^{\bar n-2}E|^2+\text{lower order
terms}\left.\vphantom{\nabla^{\bar n-2}\tau}\right\}
\operatorname{dvol}_m\,.
\end{eqnarray*}\end{enumerate}\end{theorem}

In this paper, we will establish a corresponding leading term analysis for the heat content asymptotics. We shall always
assume
$\ell$ is even; thus the lack of symmetry in the way we have written the interior contributions plays no role. Let
$\nabla$ be the connection defined by $D$ as discussed in Section~\ref{sect-1.8}. Let $D^*$ be the formal adjoint of
$D$; the associated connection $\nabla^*$ defined by $D^*$ is then the connection dual to $\nabla$ defined by the
relation
$$\nabla\phi_1\cdot\phi_2+\phi_1\cdot\nabla^*\phi_2=d(\phi_1\cdot\phi_2)\,.$$
Let 
$$\phi_1^{(\ell)}:=\nabla_\nu^\ell\phi_1|_{\partial M}\quad\text{and}\quad
\phi_2^{(\ell)}:=(\nabla_\nu^*)^\ell\phi_2|_{\partial M}$$
be the normal covariant derivatives of order $\ell$. By using the inward geodesic flow, we can
always choose coordinates $(y,r)$ near the boundary so that $\partial_r=\nu$; consequently
$$\phi^{(\ell)}=\partial_r^\ell\phi|_{\partial M}\quad\text{if}\quad D=\Delta_g\,.$$
Let
$S$ be a smooth function on the boundary. The Robin boundary operator in this more general setting is defined by the identity:
$$\BB_S^+\phi:=(\phi^{(1)}+S\phi)|_{\partial M}\,.$$
Let $\rho_{mm}^{(\ell)}:=R_{amma;m...m}$ be the
$\ell^{\operatorname{th}}$ covariant derivative of $\rho_{mm}$ restricted to
$\partial M$. Define $\Xi_\ell$ recursively for $\ell$ even by setting:
$$
\textstyle\Xi_2=-2\pi^{-1/2}\frac23\quad\text{and}
\quad\Xi_\ell=\frac2{\ell+1}\Xi_{\ell-2}\quad\text{if}\quad\ell\ge4\,.
$$
\goodbreak\begin{theorem}\label{thm-1.9}
Let $\ell\ge6$ be even. Modulo lower order terms we have:
\begin{enumerate}
\item $\displaystyle\beta_\ell^{\partial M}(\phi_1,\phi_2,D,\mathcal{B}^-)
   =\int_{\partial
M}\left\{\Xi_\ell(\phi_1^{(\ell)}\phi_2+\phi_1\phi_2^{(\ell)})+\ell\cdot\Xi_\ell\phi_1\phi_2E^{(\ell-2)}\right.$
\medbreak
$+0\cdot(\phi_1^{(\ell-1)}\phi_2^{(1)}+\phi_1^{(1)}\phi_2^{(\ell-1)})
+(\ell-2)\Xi_\ell(\phi_1^{(1)}\phi_2+\phi_1\phi_2^{(1)})E^{(\ell-3)}$
\medbreak$+
0\cdot\phi_1^{(1)}\phi_2^{(1)}E^{(\ell-4)}+\frac12(\ell-2)\Xi_\ell\phi_1\phi_2\rho_{mm}^{(\ell-2)}+...\left.\vphantom{\phi_1^{(\ell}}\right\}
\operatorname{dvol}_{m-1}$.
\smallbreak\item
$\displaystyle\beta_\ell^{\partial M}(\phi_1,\phi_2,D,\mathcal{B}_S^+)
   =\int_{\partial M}\left\{0(\phi_1^{(\ell)}\phi_2+\phi_1\phi_2^{(\ell)})+0\cdot\phi_1\phi_2E^{(\ell-2)}\right.$
\medbreak
$-\Xi_\ell(\phi_1^{(\ell-1)}\phi_2^{(1)}+\phi_1^{(1)}\phi_2^{(\ell-1)})
-\Xi_\ell(\phi_1^{(1)}\phi_2+\phi_1\phi_2^{(1)})E^{(\ell-3)}$
\medbreak
$+
(2-\ell)\Xi_\ell\phi_1^{(1)}\phi_2^{(1)}E^{(\ell-4)}-\Xi_\ell S(\phi_1^{(\ell-1)}\phi_2+\phi_1\phi_2^{(\ell-1)})$
\medbreak
$-\Xi_\ell S(\phi_1^{(\ell-2)}\phi_2^{(1)}+\phi_1^{(1)}\phi_2^{(\ell-2)})
-2\cdot\Xi_\ell S(\phi_1\phi_2^{(1)}+\phi_1^{(1)}\phi_2)E^{(\ell-4)}$
\medbreak
$+0\cdot\phi_1\phi_2\rho_{mm}^{(\ell-2)}+...\left.\vphantom{\phi_1^{(\ell)}}\right\}\operatorname{dvol}_{m-1}$.
\end{enumerate}\end{theorem}

\subsection{Outline of the paper} In Section~\ref{sect-2} we will prove
Theorem~\ref{thm-1.1} and Theorem~\ref{thm-1.5}. In Section~\ref{sect-3}, we use
Theorem~\ref{thm-1.9} to establish Theorem~\ref{thm-1.6} and Theorem~\ref{thm-1.7}.
In Section~\ref{sect-4}, we use Theorem~\ref{thm-1.8} to demonstrate Theorem~\ref{thm-1.3}. Theorem~\ref{thm-1.9} is new
and is proved in Section~\ref{sect-5} by extending functorial methods employed in \cite{BeDeGi93,BeGi94}.  In
Section~\ref{sect-6}, we establish Theorem~\ref{thm-1.2}. We conclude the paper in
Section~\ref{sect-7} by demonstrating Theorem~\ref{thm-1.4}.

\section{Local invariants in the real analytic setting}\label{sect-2}

Let $\alpha:=(\alpha_1,\dots,\alpha_m)$ be a non-trivial multi-index. We define:
$$|\alpha|:=\alpha_1+\dots+\alpha_m,\quad 
\partial_x^\alpha:=(\partial_{x_1})^{\alpha_1}\dots(\partial_{x_m})^{\alpha_m},\quad
g_{ij/\alpha}:=\partial_x^\alpha g_{ij}\text{ for }|\alpha|>0\,.$$
In any local system of
coordinates, the Riemannian volume form on
$M$ is given by:
$$\operatorname{dvol}_m=gdx,\quad\text{where}\quad g:=\sqrt{\det(g_{ij})}\,.$$
Let $g^{ij}$ be the inverse matrix;  this gives the components of the dual metric on the cotangent bundle.
Since the heat trace and heat content
asymptotics are given by suitable local formulae, Theorem~\ref{thm-1.1} and Theorem~\ref{thm-1.5} will follow from the
following result:
\goodbreak\begin{theorem}
Let $\mathcal{E}_n$ be a local interior invariant which is homogeneous of degree $n$ in the jets of the metric and
a finite (possibly empty) collection $\{\phi_1,...\}$ of additional smooth functions.
Let $\mathcal{F}_{n-1}$ be a local boundary invariant which is homogeneous of degree $n-1$ in the jets of the metric and
a finite (possibly empty) collection $\{\phi_1,...\}$ of additional smooth functions.
Let $(M,g)$ be a compact real analytic manifold of dimension $m$ with real analytic (possibly empty) boundary $\partial M$ so that the metric $g$
is real analytic and so that
the collection $\{\phi_1,...\}$ is real analytic.
 There
exists a constant $C=C(M,g,\phi_1,...)>0$ (which is independent of the choice of $\mathcal{E}_n$ and of $\mathcal{F}_n$)
and there exist constants
$\kappa(\mathcal{E}_n)>0$ and $\kappa(\mathcal{F}_{n-1})>0$ (which are independent of the choice of $(M,g,\phi_1,...))$ so
that 
\begin{eqnarray*}
&&\left|\int_M\mathcal{E}_n(x,g,\phi_1,...)\operatorname{dvol}_m\right|\le\kappa(\mathcal{E}_n)C^n\cdot
\operatorname{vol}_m(M,g),\\
&&\left|\int_{\partial
M}\mathcal{F}_{n-1}(y,g,\phi_1,...)\operatorname{dvol}_{m-1}\right|\le\kappa(\mathcal{F}_{n-1})
C^{n-1}\cdot\operatorname{vol}_{m-1}(\partial
M,g)\,.
\end{eqnarray*}
The constant $C(M,g,\phi_1,...)$ may be chosen so that 
$$C(M,c^2g,\phi_1,...)=c^{-n}C(M,g,\phi_1,...)\,.$$
\end{theorem}

\goodbreak\begin{proof} Suppose first that the boundary of $M$ is empty. For each point $P$ of $M$, there exists $\varepsilon(P)>0$ so
the exponential map defines a real analytic geodesic coordinate ball of radius $\varepsilon(P)$ about $P$. Let $\mathcal{K}$
be a compact neighborhood of the identity in the space of all symmetric $m\times m$ matrices. Since $g_{ij}=\delta_{ij}$ at
the center of such a geodesic coordinate ball, by shrinking $\varepsilon(P)$ if necessary, we may assume that the matrix
$(g_{ij})$ belongs to $\mathcal{K}$ for any point of the coordinate ball of radius $\varepsilon(P)$. Since we are working in the real analytic
category and since
$\{g_{ij},\phi_1,...\}$ are real analytic near $P$ there exists a
$C=C(P,M,g,\phi_1,...)$ so that again by shrinking $\varepsilon(P)$ if necessary we have that
\begin{equation}\label{eqn-2.a}
|d_x^\alpha g_{ij}|\le C^{|\alpha|}|\alpha|!\quad\text{and}\quad
|d_x^\alpha\phi_\mu|\le C^{|\alpha|}|\alpha|!\quad\text{on}\quad B_{\varepsilon(P)}(P)
\end{equation}
for any multi-index $\alpha$. We cover $M$ by a finite number of such coordinate balls about points $(P_1,...)$ and
set $C(M,g,\phi_1,...)=\max_\nu C(P_\nu,M,g,\phi_1,...)$. Since
$\mathcal{E}$ is a local invariant, we may expand: 
\begin{equation}\label{eqn-2.b}
\mathcal{E}(x,g)=\sum e_{\vec\alpha,\vec\beta}(g_{ij}(x))(\partial_x^{\alpha_1}g_{i_1j_1})
...(\partial_x^{\alpha_a}g_{i_aj_a})\cdot(d_x^{\beta_1}\phi_{k_1})...(d_x^{\beta_b}\phi_{k_b})
\end{equation}
where in this sum we have the relations:
$$
   |\alpha_1|+...|\alpha_a|+|\beta_1|+...+|\beta_b|=n,\ 0<|\alpha_1|,\ ...,\  0<|\alpha_a|\,.
$$
Since $e_{\vec\alpha,\vec\beta}$ is continuous on the compact neighborhood $\mathcal{K}$ of the identity $\delta$, we may
bound 
$$
|e_{\vec\alpha,\vec\beta}(g_{ij}(x))|\le E_{\vec\alpha,\vec\beta }\quad\text{uniformly on}\quad\mathcal{K}\,.
$$
Combining the estimates of Equation (\ref{eqn-2.a}) with the estimates given above and summing over
$(\vec\alpha,\vec\beta)$ in Equation (\ref{eqn-2.b}) yields an estimate of the desired form after integration. Since
$\mathcal{E}_n$ is homogeneous of degree $n$, it follows that
$$\mathcal{E}_n(x,c^2g,\phi_1,...)=c^{-n}\mathcal{E}_n(x,g,\phi_1,...)\,.$$
The desired rescaling behaviour of the constant $C(M,g,\phi_1,...)$ now follows.

If the boundary of $M$ is non-empty, we must also choose
suitable coordinate charts near
$\partial M$. If
$Q\in\partial M$, we consider the geodesic ball $B_\varepsilon^{\partial M}(Q)$ of radius $\varepsilon$ in $\partial M$
about
$Q$         relative to the restriction of the metric to the boundary and we shall let
$\tilde B_{\varepsilon,\iota}(Q):=[0,\iota)\times B_{\varepsilon(Q)}^{\partial M}(Q)$ for some $\iota>0$ be defined using the
inward geodesic flow so that the curves
$r\rightarrow(r,Q)$ are unit speed geodesics perpendicular to the boundary. Again, by shrinking
$\varepsilon$ and
$\iota$, we may achieve the estimates of Equation (\ref{eqn-2.a}) uniformly on $\tilde B_{\varepsilon,\iota}(Q)$. We
cover
$M$ by a finite number of coordinate charts
$B_{\varepsilon}(P)$ for $P\in\operatorname{int}(M)$ and $\tilde B_{\iota,\varepsilon}(Q)$ for $Q\in\partial M$. The desired
estimate for $\mathcal{E}_n$ now follows. To study the invariant
$\mathcal{F}_{n-1}$, we cover $\partial M$ by a finite number of coordinate charts $\tilde B_{\iota,\varepsilon}(Q)$ for
$Q\in\partial M$ and argue as above.\end{proof}

\section{Leading Terms in the Heat Content Asymptotics}\label{sect-3}

We shall omit the proof of the following result as it is well known.

\goodbreak\begin{lemma}\label{lem-3.1}
\ \begin{enumerate}
\item Let $k\ge1$ be given, let constants $\gamma_\ell>0$ for $\ell\ge k$ be given,
and let $\epsilon>0$ be given. Let $(M,g)$ be a smooth Riemannian manifold with non-empty boundary $\partial M$. There
exists a smooth function $\Phi$ on $M$ so that $||\Phi||_{k-1}<\varepsilon$ and so that
$$\Phi^{(\ell)}=\psi(y)\gamma_\ell\quad\text{for}\quad \ell\ge k\,.$$
\item Let $k\ge1$ be given, let $C>0$ be given,
and let $\epsilon>0$ be given. There exists a smooth function $f$ on $M:=[0,1]$
with $||f||_{k-1}<\varepsilon$ and $\int_M|\partial_x^kf|^2dx\ge
C$.\end{enumerate}
\end{lemma}

\goodbreak\begin{proof}[The proof of Theorem~\ref{thm-1.6} and of Theorem~\ref{thm-1.7}]
Let $k\ge3$ be given, let constants $C_{\bar\ell}>0$ for $\bar\ell\ge k$ be given,
and let $\epsilon>0$ be given. Let $(M,g)$ be a smooth compact Riemannian
manifold of dimension $m\ge2$ with non-trivial boundary. 
We first take $\BB=\BB^-$ to consider Dirichlet boundary conditions. Let $\phi_1$ be a smooth initial temperature
 and let $\phi_2$ be a smooth specific heat with $\BB^-\phi_2\ne0$. 
Since $\phi_2$ does not vanish identically on the boundary, there exists a smooth function $\psi$ on $\partial
M$ so
$$
\int_{\partial M}\psi\phi_2\operatorname{dvol}_{m-1}=1\,.
$$
Let $\{\gamma_1,...\}$ be a sequence of constants, to be determined presently. For $\nu\ge k$, let
$$\Phi_\nu(y,r)=\sum_{j=k}^{\phantom{.}\nu}\frac{ r^{2j}}{(2j)!}\gamma_j\psi(y)\text{ near }\partial M\,.$$
Since $\beta_{2\bar\ell}$ is given by a local formula of degree $2\bar\ell$, only the constants $\gamma_1$, ...,
$\gamma_{\bar\ell}$ play a role in the computation of $\beta_{2\bar\ell}^{\partial M}$, i.e. 
$$
\beta_{2\bar\ell}^{\partial M}(\Phi_\mu+\phi_1,\phi_2,\Delta_g,\BB)
   =\beta_{2\bar\ell}^{\partial M}(\Phi_{\bar\ell}+\phi_1,\phi_2,\Delta_g,\BB)\quad\text{if}\quad\mu\ge\bar\ell\,.
$$

We take $\Phi_{k-1}=0$. Since $\Xi_{2\bar\ell}\ne0$, we can
recursively choose the constants $\gamma_{\bar\ell}$, and hence the functions $\Phi_{\bar\ell}$, for $\bar\ell\ge k$ so
$$\Xi_{2\bar\ell}\cdot\gamma_{\bar\ell}=C_{\bar\ell}-\beta_{2\bar\ell}^{\partial
M}(\Phi_{\bar\ell-1}+\phi_1,\phi_2,\Delta_g,\BB)\quad\text{for}
\quad\bar\ell\ge k$$
and apply Theorem~\ref{thm-1.9} to see:
$$\beta_{2\bar\ell}^{\partial M}(\Phi_{2\bar\ell}+\phi_1,\phi_2,\Delta_g,\BB^-)=C_{\bar\ell}\,.$$
We complete the proof of Theorem~\ref{thm-1.6} (1) by using Lemma~\ref{lem-3.1} to choose $\Phi$ with
$||\Phi||_{2k-1}<\varepsilon$ such that
$$\Phi^{(j)}=\left\{\begin{array}{lll}
0&\text{if}&j<2k\text{ or if }j\text{ is odd}\\
\gamma_{\bar\ell}&\text{if}&j=2\bar\ell\text{ for }\bar\ell\ge k\end{array}\right\}.$$

To prove Assertion (2) of Theorem~\ref{thm-1.6}, we use Assertion (2) of Theorem~\ref{thm-1.9} and examine
the term $-\Xi_{2\bar\ell}\phi_1^{(2\bar\ell-1)}\phi_2^{(1)}$; to prove Theorem~\ref{thm-1.7}, we apply Assertion (1) of
Theorem~\ref{thm-1.9} and examine the term $\frac12(2\bar\ell-2)\Xi_{2\bar\ell}\phi_1\phi_2\rho_{mm}^{(2\bar\ell-2)}$.
As apart from these minor changes the proof is exactly the same as that given above, we shall omit details
in the intersts of brevity.\end{proof}

\section{Leading terms in the heat trace asymptotics}\label{sect-4}

\subsection{Proof of Theorem~\ref{thm-1.3} (1)}
We set $E=0$
 and $\Omega=0$ in Theorem~\ref{thm-1.8} to study the Laplacian and see thereby that there exists a non-zero constant
$d_n$ so:
\begin{eqnarray*}
a_{2\bar n}(\Delta_g)&=&d_{\bar n}\int_M\left\{(\bar n^2-\bar n-1)|\nabla^{\bar
n-2}\tau|^2+2|\nabla^{\bar n-2}\rho|^2\right.\\ &+&\left.Q_{\bar n,m}(R,\nabla
R,...,\nabla^{\bar n-3}R)\vphantom{\nabla^{\bar n-2}\rho|^2}\right\}\operatorname{dvol}_m\,.
\end{eqnarray*}
Let $\varepsilon>0$ be given. We restrict to a
single geodesic ball $B$ of radius $3\delta$ for some $\delta>0$ about a point $P$.
 Let $\theta$ be a
plateau function so that $\theta=1$ for $|x|<\delta$ and $\theta=0$ for $|x|>2\delta$.
We shall define the functions $f_k$, $f_{k+1}$, ... recursively and consider the conformal deformation:
$$g_\mu:=e^{\theta(x)(2f_k(x_1)+...+2f_\mu(x_1))}g\,.$$
Let $k\ge3$. Choose $0<\delta_\mu^1$ for $k\le\mu$ so that
$||f_\mu||_{\mu-1}\le\delta_\mu^1$ for $k\le\mu$ implies:
\begin{constraint}\ \begin{enumerate}
\item $f_\infty:=\lim_{\mu\rightarrow\infty}\{f_k+\dots+f_\mu\}$ converges in the $C^\ell$ topology for any $\ell$.
\smallbreak\item $g_\infty:=\lim_{\mu\rightarrow\infty}g_\mu$ converges in the $C^\ell$ topology for any $\ell$.
\smallbreak\item $||f||_{k-1}<\varepsilon$.
\smallbreak\item $||g_\mu-g_{\mu+1}||_{\mu}<2^{-\mu}\varepsilon$ for any $\mu$.
\end{enumerate}\end{constraint}

A-priori, one must consider jets of degree $2\bar n$ in computing
$a_{2\bar n}(\Delta_g)$ (and in fact this is the case when considering
the local heat asymptotic coefficients of Equation~(\ref{eqn-1.c})). However, by Theorem~\ref{thm-1.8}, only the jets of
the metric to degree
$\bar n$ play a role in the computation of the integrated invariants, $a_{2\bar n}$.
\begin{constraint}
\rm Choose $0<\delta_\mu^2<\delta_\mu^1$ for $k\le\mu$ so
$||f_\mu||_{\mu-1}\le\delta_\mu^2$ for $k\le\mu$ implies:
\begin{enumerate}
\item $|a_{2\bar n}(\Delta_{g_{\mu-1}})-a_{2\bar n}(\Delta_{g_\mu})|<2^{-\mu}$ for  $3\le k\le\bar n<\mu$.
\smallbreak\item $|a_{2\bar n}(\Delta_{g_\mu})|-1\le |a_{2\bar n}(\Delta_{g_\infty})|$ for $3\le k\le\bar n$.
\end{enumerate}\end{constraint}

The polynomial $Q_{\bar n,m}(\cdot)$ involves lower order derivatives of the metric.
\begin{constraint}
\rm Choose $0<\delta_\mu^3<\delta_\mu^2$ for $k\le\mu$ so that
$||f_\mu||_{\mu-1}\le\delta_\mu^3$ for $k\le\mu$ implies
there are constants $C_\mu^1=C_\mu^1(f_k,\dots,f_{\mu-1})$ depending only on the choices made previously so
\begin{eqnarray*}
|a_{2\mu}(\Delta_{g_\mu})|&\ge&|d_\mu|\int_M
\left\{|2\nabla^{\mu-1}\tau_{g_\mu}|^2+(\mu^2-\mu-1)|\nabla^{n-1}\rho|^2\right\}\operatorname{dvol}_m
-C_\mu^1\\
&\ge&|d_\mu|\int_{B_\delta}
\left\{|2\nabla^{\mu-1}\tau_{g_\mu}|^2\right\}\operatorname{dvol}_m-C_\mu^1\,.
\end{eqnarray*}
\end{constraint}

On $B_\delta$, the plateau function $\theta$ is identically $1$ and we have:
$$g_{\mu}=e^{2f_\mu}g_{\mu-1}\,.$$
From this it follows that
$$
\nabla^{\bar n-2}\tau=(m-1)\partial_{x_1}^{\bar n}f_\mu+\text{lower order terms}\,.
$$
Since $g_{ij}$ is in a compact neighborhood of $\delta_{ij}$, we may estimate:
\begin{equation}\label{eqn-4.a}
||\nabla^{\bar n-2}\tau_{g_n}||^2(P)\ge|\partial_{x_1}^{\bar n-2}\tau|^2
=|\partial_{x_1}^{\bar n}f_{\bar n}|^2+\text{lower order terms}\,.
\end{equation}
\begin{constraint}\label{con-4}
\rm Choose $0<\delta_\mu^4<\delta_\mu^3$ for $k\le\mu$ where
$\delta_\mu^4=\delta_\mu^4(f_k,\dots,f_{\mu-1})$ depends on the choices made previously so that
$||f_\mu||_{\mu-1}\le\delta_\mu^4$ for $k\le\mu$ implies there are constants
$C_\mu^2=C_\mu^2(f_k,\dots,f_{\mu-1})$ depending only on the choices made previously so
$$\int_{B_{\delta_\mu^4}}|\nabla^{n-2}\tau_{g_\mu}|^2\operatorname{dvol}_m\ge
\int_{B_{\delta_\mu^4}}|\partial_{x_1}^\mu f_\mu|^2\operatorname{dvol}_m-C_\mu^2\,.$$
\end{constraint}
Theorem~\ref{thm-1.1} (1) now follows from Lemma~\ref{lem-3.1} (2). We can choose recursively $f_\mu$ subject to the
constraints given above so that $||f_\mu||_{\mu-1}$ is arbitrarily small and so that
$\int_{B_{\delta_\mu^4}}|\partial_{x_1}^\mu f_\mu|^2\operatorname{dvol}_m$ is arbitrarily large.\hfill\qed

\subsection{The proof of Theorem~\ref{thm-1.1} (2)} Let $(M,g)$ be a hypersurface in $\mathbb{R}^m$. We fix
$P\in M$. After applying a rigid body motion, we may assume that
$P=0$ and that the normal to $M$ at $P$ is given by
$e_{m+1}:=(0,\dots,0,1)$. Thus we may write $M$ as a graph over
the ball $B_{3\delta}$ in $\mathbb{R}^m$ in the form
$x\rightarrow(x,f_0(x))$ where $f_0(P)=0$ and $df_0(P)=0$. Let
$\theta$ be a plateau function which is $1$ for $|x|\le\delta$ and
$0$ for $|x|\ge\delta$. We shall consider the perturbed
hypersurface defined near $P$ by
$x\rightarrow(x,f_0(x)+\theta(x)(f_k(x)+\dots))$ where
$f_\mu(P)=0$ and $df_\mu(P)=0$. This hypersurface agrees with the
original hypersurface away from $P$. We shall need to establish an
analogue of Equation (\ref{eqn-4.a}). The remainder of the
analysis will be similar to that performed in the proof of Theorem~\ref{thm-1.1} (1), and will therefore be omitted.

Suppose we have a hypersurface in the form $\Psi(x):=(x,F(x))$ where $F(0)=0$ and $dF(0)=0$. Let
$F_i:=\partial_{x_i}F$,
$F_{ij}:=\partial_{x_i}\partial_{x_j}F$, and so forth. We compute:
\begin{eqnarray*}
&&\Psi_*(\partial_{x_i})=e_i+F_ie_{m+1},\\
&&g_{ij}=\delta_{ij}+F_iF_j,\\
&&\Gamma_{jkl}=\textstyle\frac12\{F_{jk}F_l+F_{jl}F_k+F_{jk}F_l+F_{kl}F_j-F_{jl}F_k-F_{kl}F_j\}=F_{jk}F_l,\\
&&\Gamma_{jk}{}^l=g^{ln}F_{jk}F_n,\\
&&R_{ijk}{}^l=g^{ln}\{F_{jk}F_{in}-F_{ik}F_{jn}\}+\text{lower order terms},
\end{eqnarray*}
where the lower order terms are either $4^{\operatorname{th}}$ order in the $1$-jets or linear in the $2$-jets and
quadratic in the $1$-jets. We suppose $F=F_{\mu-1}+f_\mu$ where we set $f_\mu=\varepsilon_\mu\cos(
a_\mu x^1)\cos(b_\mu x^2)$.
\begin{eqnarray*}
&&\tau=4\varepsilon_\mu a_\mu^2b_\mu^2\{\cos^2(a_\mu x^1)\cos^2(b_\mu x^1)-\sin^2(a_\mu x^1)\sin^2(b_\mu
   x^1)\}+\dots,\\
&&|\nabla^{\mu-2}\tau|^2=4\varepsilon_\mu a_\mu^4b_\mu^\mu|\cos^2(a_\mu x^1)
\cos^2(b_\mu x^1)-\sin^2(a_\mu x^1)\sin^2(b_\mu
   x^1)|^2+\dots,
\end{eqnarray*}
where we have omitted lower order terms either involving
$\varepsilon^2$ or not multiplied by the appropriate power of
$a_\mu^4b_\mu^\mu$. To simplify matters, we suppose $\delta=\pi$
and that $a_\mu$ and $b_\mu$ are non-zero integers. We use the
fact that we are dealing with periodic functions to compute:
\begin{eqnarray*}
&&\int_{x^1=-\pi}^\pi\int_{x^2=-\pi}^\pi |\cos^2(a_\mu x^1)\cos^2(b_\mu x^2)-\sin^2(a_\mu x^1)\sin^2(b_\mu x^2)|^2
dx^2dx^1\\ &=&a_\mu^{-1}b_\mu^{-1}\int_{x^1=-a_\mu\pi}^{a_\mu\pi}\int_{x^2=-b_\mu\pi}^{b_\mu\pi}
|\cos^2(x^1)\cos^2(x^2)-\sin^2(x^1)\sin^2(x^2)|^2 dx^2dx^1\\
&=&a_\mu^{-1}b_\mu^{-1}a_\mu b_\mu\int_{x^1=-\pi}^{\pi}\int_{x^2=-\pi}^{\pi}
|\cos^2(x^1)\cos^2(x^2)-\sin^2(x^1)\sin^2(x^2)|^2 dx^2dx^1\\
&=&(2\pi)^2\,.
\end{eqnarray*}
We shall take $b_\mu=a_\mu^\mu$, take $a_\mu$ large, and take $\varepsilon_\mu$
appropriately small to complete the proof.\hfill\qed

\section{Leading terms in the heat content asymptotics}\label{sect-5}
This section is devoted to the proof of Theorem~\ref{thm-1.9}. Let $D$ be an operator of Laplace type on a compact
smooth Riemannian manifold $(M,g)$ with non-empty boundary. We adopt the notation established in Section~\ref{sect-1.8}
and in Section~\ref{sect-1.9}. We shall always take $S$ to be real in defining the Robin boundary operator. One then has
the symmetry
\begin{equation}\label{eqn-5.a}
\beta(\phi_1,\phi_2,D,\BB)(t)=\beta(\phi_2,\phi_1,D^*,\BB)(t)\,.
\end{equation}
If
$\ell$ is even, the lack of symmetry in the way we expressed the interior terms plays no role and thus Equation
(\ref{eqn-5.a}) yields:
\begin{equation}\label{eqn-5.b}
\beta_{2\bar\ell}^{\partial M}(\phi_1,\phi_2,D,\BB)=\beta_{2\bar\ell}^{\partial
M}(\phi_2,\phi_1,D^*,\BB)\,.
\end{equation}

Let indices
$\{a,b\}$ range from
$1$ to
$m-1$ and index the tangential coordinates $(y^1,\dots,y^{m-1})$ in an adapted coordinate system
such that $\partial_r$ is the inward unit geodesic normal. We then have
$$ds^2=g_{ab}(y,r)dy^a\circ dy^b+dr\circ dr\,.$$
We define the {\it second fundamental} form by setting:
$$
  L_{ab}:=g(\nabla_{\partial_{y_a}}\partial_{y_b},\partial_r)=-\textstyle\frac12\partial_rg_{ab}\,.
$$
Results of \cite{BeDeGi93,BeGi94} yield the following formulae which will form the starting point for our analysis: 
\goodbreak\begin{lemma}\label{lem-5.1}
Adopt the notation established above. Then
\begin{enumerate}\item
$\beta_0^{\partial M}(\phi_1,\phi_2,D,\BB^-)=-\frac{2}{\sqrt\pi}\int_{\partial
M}\phi_1\phi_2\operatorname{dvol}_{m-1}$.
\smallbreak\item $\beta_0^{\partial M}(\phi_1,\phi_2,D,\BB_S^+)=\phantom{-}0$.
\smallbreak\item $\beta_2^{\partial M}(\phi_1,\phi_2,D,\BB^-)
     =-\frac{2}{\sqrt\pi}\int_{\partial
M}\left\{\frac23\left(\phi_1^{(2)}\phi_2+\phi_1\phi_2^{(2)}\right)+\phi_1\phi_2E\right.$
\smallbreak$\qquad-\phi_{1;a}\phi_{2;a}-\frac23L_{aa}\left(\phi_1^{(1)}\phi_2+\phi_1\phi_2^{(1)}\right)$
\smallbreak$\qquad+\left(\frac1{12}L_{aa}L_{bb}-\frac16L_{ab}L_{ab}-\frac16\rho_{mm}\right)\phi_1\phi_2\big\}
\operatorname{dvol}_{m-1}$.
\smallbreak\item
$\beta_2^{\partial M}(\phi_1,\phi_2,D,\BB_S^+)
     =\phantom{-}\frac{2}{\sqrt\pi}\int_{\partial
M}\frac23(\phi_1^{(1)}+S\phi_1)(\phi_2^{(2)}+S\phi_2)\operatorname{dvol}_{m-1}$.
\end{enumerate}\end{lemma}

We begin the proof of Theorem \ref{thm-1.9} by expressing
$\beta_\ell^{\partial M}$, modulo lower order terms, in terms of certain invariants involving maximal derivatives with
unknown but universal coefficients; the symmetry of Equation (\ref{eqn-5.b}) plays a crucial role in our
analysis. Standard arguments (see \cite{BeDeGi93}) show
the coefficients in the following expressions are independent of the underlying dimension of
the manifold:\medbreak
$\displaystyle\beta_\ell^{\partial M}(\phi_1,\phi_2,D,\mathcal{B}^-)
   =\int_{\partial M}\left\{c_{\ell,1}^-(\phi_1^{(\ell)}\phi_2+\phi_1\phi_2^{(\ell)})
+c_{\ell,2}^-(\phi_1^{(\ell-1)}\phi_2^{(1)}+\phi_1^{(1)}\phi_2^{(\ell-1)})\right.$
\medbreak
$\displaystyle\quad\left.+e_{\ell,1}^-\phi_1\phi_2E^{(\ell-2)}+e_{\ell,2}^-(\phi_1^{(1)}\phi_2+\phi_1\phi_2^{(1)})E^{(\ell-3)}+
e_{\ell,3}^-\phi_1^{(1)}\phi_2^{(1)}E^{(\ell-4)}+\right.$
\medbreak
$\quad\left.+r_\ell^-\phi_1\phi_2\rho_{mm}^{(\ell-2)}+...\right\}\operatorname{dvol}_{m-1}$,
\medbreak
$\displaystyle\beta_\ell^{\partial M}(\phi_1,\phi_2,D,\mathcal{B}_S^+)
   =\int_{\partial M}\left\{c_{\ell,1}^+(\phi_1^{(\ell)}\phi_2+\phi_1\phi_2^{(\ell)})
+c_{\ell,2}^+(\phi_1^{(\ell-1)}\phi_2^{(1)}+\phi_1^{(1)}\phi_2^{(\ell-1)})\right.$
\medbreak
$\displaystyle\quad\left.+e_{\ell,1}^+\phi_1\phi_2E^{(\ell-2)}+e_{\ell,2}^+(\phi_1^{(1)}\phi_2+\phi_1\phi_2^{(1)})E^{(\ell-3)}+
e_{\ell,3}^+\phi_1^{(1)}\phi_2^{(1)}E^{(\ell-4)}\right.$
\medbreak
$\displaystyle\quad+d_{\ell,1}^+S(\phi_1^{(\ell-1)}\phi_2+\phi_1\phi_2^{(\ell-1)})
   +d_{\ell,2}^+S(\phi_1^{(\ell-2)}\phi_2^{(1)}+\phi_1^{(1)}\phi_2^{(\ell-2)})$
\medbreak
$\displaystyle\quad
 \left.+d_{\ell,3}^+S(\phi_1\phi_2^{(1)}+\phi_1^{(1)}\phi_2)E^{(\ell-4)}
+d_{\ell,5}^+S\phi_1\phi_2E^{(\ell-3)}+r_\ell^+\phi_1\phi_2\rho_{mm}^{(\ell-2)}\right.$
\medbreak
$\displaystyle\quad\left.+...\vphantom{c_{\ell,1}^+\phi_1^{\ell}}\right\}\operatorname{dvol}_{m-1}$.

\medbreak 
We will determine all the coefficients except $d_{\ell,5}^+$ in what follows. Recall that
$$\Xi_2=-2\pi^{-1/2}\frac23\quad\text{and}\quad\Xi_\ell=\frac2{\ell+1}\Xi_{\ell-2}\,.$$
\goodbreak\begin{lemma}\label{lem-5.2}
Let $\ell\ge4$ be even. Let $\BB=\BB^-$ or $\BB=\BB_S^+$.\begin{enumerate}
\smallbreak\item Let $D$ be self-adjoint with respect to the boundary conditions defined by $\BB$. If
$\mathcal{B}\phi_1=0$, then
$\beta_\ell^{\partial M}(\phi_1,\phi_2,D,\BB)=\textstyle\frac2{\ell+1}\beta_{\ell-2}(\phi_1^{(2)}+E,\phi_2,D,\BB)$.
\smallbreak\item $c_{\ell,1}^-=\Xi_\ell$, $c_{\ell,2}^-=0$, $c_{\ell,1}^+=0$, and $c_{\ell,2}^+=-\Xi_\ell$.
\smallbreak\item $e_{\ell,2}^-=(\ell-2)\Xi_\ell$, $e_{\ell,3}^-=0$, $e_{\ell,1}^+=0$, $e_{\ell,2}^+=-\Xi_\ell$, and
$r_\ell^+=0$.
\smallbreak\item $d_{\ell,1}^+=d_{\ell,2}^+=-\Xi_\ell$.
\end{enumerate}\end{lemma}

\goodbreak\begin{proof} We follow \cite{BeDeGi93} to derive Assertion (1) as follows. Let
$\{\lambda_\mu,\phi_\mu\}$ be a complete spectral resolution of $D_\BB$. Here $\{\phi_\mu\}$ is a 
complete orthonormal basis for $L^2(M)$ of smooth functions with $D\phi_\mu=\lambda_\mu\phi_\mu$
and $\BB\phi_\mu=0$. Let
$$\gamma_\mu^D(f):=\int_Mf\phi_\mu\operatorname{dvol}_m$$ be the associated Fourier coefficients. Then
$$\beta(\phi_1,\phi_2,D,\BB)(t)=\sum_{\mu=1}^\infty e^{-t\lambda_\mu}\gamma_\mu^D(\phi_1)\gamma_\mu^D(\phi_2)\,.$$
 If
$\BB \phi_1=0$, then
$$\gamma_\mu^D(D\phi_1)=\int_M D\phi_1\cdot\phi_\mu\operatorname{dvol}_m=\int_M\phi_1\cdot
D\phi_\mu\operatorname{dvol}_m=\lambda_\mu\gamma_\mu^D(\phi_1)\,.$$
Consequently we have that:
\begin{eqnarray*}
&&\beta(D\phi_1,\phi_2,D,\BB)(t)\\
&\sim&\sum_{n=0}^\infty \frac{(-t)^n}{n!}\int_MD^{n+1}\phi_1\cdot \phi_2\operatorname{dvol}_m
+\sum_{k=0}^\infty t^{(k+1)/2}\beta_k^{\partial M}(D\phi_1,\phi_2,D,\BB)\\
&=&\sum_{\mu=1}^\infty e^{-t\lambda_\mu}\gamma_\mu^D(D\phi_1)\gamma_\mu^D(\phi_2)=\sum_{\mu=1}^\infty \lambda_\mu
e^{-t\lambda_\mu}\gamma_\mu^D(\phi_1)\gamma_\mu^D(\phi_2)\\ &=&-\frac{\partial}{\partial t}
\sum_{\mu=1}^\infty e^{-t\lambda_\mu}\gamma_\mu^D(\phi_1)\gamma_\mu^D(\phi_2)=-\frac{\partial}{\partial
t}\beta(\phi_1,\phi_2,D,\BB)(t)\\ &\sim&\sum_{j=1}^\infty \frac{(-t)^{j-1}}{(j-1)!}\int_MD^j\phi_1\cdot
\phi_2\operatorname{dvol}_m-\sum_{\ell=0}^\infty\frac {\ell+1}2t^{(\ell-1)/2}\beta_\ell^{\partial M}(\phi_1,\phi_2,D,\BB)\,.
\end{eqnarray*}
The asymptotics defined by the interior integrals are the same. We note that $-D\phi_1=\phi_1^{(2)}+E\phi_1$. We set $k=\ell-2$ and
equate the asymptotics defined by the boundary integrals to establish Assertion (1).

If $\ell=2$, then the relations of Assertion (2) would follow from Lemma~\ref{lem-5.1}
modulo the caveat that we have but a single term $c_{\ell,2}^\pm \phi_1^{(1)}\phi_2^{(1)}$ rather than 2 distinct terms in
that setting. This will let us apply the recursion relation of Assertion (1) even if $\ell=4$. Let $\phi_1|_{\partial
M}=\phi_1^{(1)}|_{\partial M}=0$. We set $E=0$ and consider
$c_{\ell,1}^\pm \phi_1^{(\ell)}\phi_2$ and $c_{\ell,2}^\pm \phi_1^{(\ell-1)}\phi_2^{(1)}$. These terms arise in
$\beta_{\ell-2}(\phi_1^{(2)},\phi_2,D,\BB)$ only from the corresponding terms
$c_{\ell-1,1}^\pm (\phi_1^{(2)})^{(\ell-2)}\phi_2$ and $c_{\ell-1,2}^\pm(\phi_1^{(2)})^{(\ell-3)}\phi_2^{(1)}$.
Assertion (2) now follows from the recursion relation
$$\textstyle c_{\ell,1}^\pm=\frac2{\ell+1}c_{\ell-2,1}^\pm\quad\text{and}\quad
c_{\ell,2}^\pm=\frac2{\ell+1}c_{\ell-2,2}^\pm\,.$$

To prove Assertion (3), we first take Dirichlet boundary conditions. Let $\ell\ge4$. Let $\phi_1^{(k)}|_{\partial M}=0$ for
$k\ne1$. No information is garnered concerning $e_{\ell,1}^-$ or $r_\ell^-$. The term
$e_{\ell,2}^-\phi_1^{(1)}\phi_2E^{(\ell-3)}$ arises in
$\beta_{\ell-2}(\phi_1^{(2)}+E\phi_1,\phi_2,D,\BB)$ only from the monomial
$c_{\ell,1}^-(\phi_1^{(2)}+E\phi_1)^{(\ell-2)}\phi_2$. It now follows that
$$e_{\ell,2}^-=(\ell-2)\textstyle\frac2{\ell+1}c_{\ell-2,1}^-=(\ell-2)\Xi_\ell\,.$$
Since the coefficient $c_{\ell-2,2}^-=0$, the term
$\phi_1^{(1)}\phi_2^{(1)}E^{(\ell-4)}$ does not arise in the
invariant$\frac2{\ell+1}\beta_{\ell-2}(\phi_1^{(2)}+E\phi_1,\phi_2,D,\BB_S^-)$ and thus
$$e_{\ell,3}^-=0\,.$$
Next we examine Neumann boundary conditions. We take $S=0$ and suppose $\phi_1^{(k)}|_{\partial M}=0$ for $k\ge1$. No information is
garnered concerning $e_{\ell,3}^+$. Since $c_{\ell,1}^+=0$, the term $e_{\ell,1}^+\phi_1\phi_2E^{(\ell-2)}$ and the term
$e_{\ell,2}^+\phi_1\phi_2^{(1)}E^{(\ell-3)}$ can arise in the invariant $\beta_{\ell-2}(\phi_1^{(2)}+E\phi_1,\phi_2,D,\BB)$ only
from the term
$c_{\ell,2}^+(\phi_1^{(2)}+E\phi_1)^{(\ell-3)}\phi_2^{(1)}$. We conclude
$$e_{\ell,1}^+=0\quad\text{and}\quad
e_{\ell,2}^+=\textstyle\frac2{\ell+1}c_{\ell,2}^+=-\Xi_\ell\,.
$$
The argument that $r_\ell^+=0$ is similar and is therefore omitted. This establishes Assertion (3).

To examine Assertion (4), we assume $\phi_1|_{\partial M}=\phi_1^{(1)}|_{\partial M}=0$. Again, we set $E=0$. We study the terms
$d_{\ell,1}^+S\phi_1^{(\ell-1)}\phi_2$ and $d_{\ell,2}^+S\phi_1^{(\ell-2)}\phi_2^{(1)}$. The case
$\ell=4$ is a bit exceptional as these terms arise in $\beta_2(\phi_1^{(2)},\phi_2,D,\BB)$ only
from
$2\pi^{-1/2}\frac23S(\phi_1^{(2)})^{(1)}\phi_2$ and from $2\pi^{-1/2}\frac23S(\phi_1^{(2)})\phi_2^{(1)}$. This shows that
$$d_{4,1}^+=d_{4,2}^+=\frac25\cdot\frac23\cdot2\pi^{-1/2}=-\Xi_4\,.$$
For $\ell\ge6$, these terms decouple and the recursion relation proceeds without complication to show
\medbreak\noindent
\hfill$d_{\ell,1}^+=\textstyle\frac2{\ell+1}d_{\ell-2,1}^+=-\Xi_\ell\quad\text{and}\quad
d_{\ell,2}^+=\textstyle\frac2{\ell+1}d_{\ell-2,2}^+=-\Xi_\ell\,.$\phantom{.}\hfill
\end{proof}

We can relate Neumann and Dirichlet boundary conditions. Let $M:=[0,1]$ and let $b\in C^\infty(M)$. Let
$\varepsilon\partial_r$ be the inward unit normal; $\varepsilon(0)=1$ and $\varepsilon(1)=-1$. Define:
\begin{equation}\label{eqn-5.c}
\begin{array}{llll}
A:=\partial_r+b,& A^*:=-\partial_r+b,&
D_1:=A^*A,& D_2^*:=AA^*,\\
S:=\varepsilon b,&\BB_S^+:=\varepsilon A,&E_1:=b^\prime-b^2,&E_2:=-b^\prime-b^2\,.\vphantom{\vrule height 12pt}
\end{array}
\end{equation}
Then $\BB_S^+\phi=0$ simply means $A\phi|_{\partial M}=0$. Furthermore $E_i$ is the endomorphism defined by $D_i$.

\goodbreak\begin{lemma}
Adopt the notation established above. Let $\ell\ge6$ be even.
\begin{enumerate}
\smallbreak\item $\beta_\ell^{\partial
M}(\phi_1,\phi_2,D_1,\BB_S^+)=-\frac2{\ell+1}\beta_{\ell-2}(A\phi_1,A\phi_2,D_2,\BB^-)$.
\smallbreak\item $e_{\ell,1}^-=\ell\cdot\Xi_\ell$,
$e_{\ell,3}^+=(2-\ell)\Xi_\ell$, $d_{\ell,3}^+=-2\cdot\Xi_\ell$.
\end{enumerate}
\end{lemma}

\goodbreak\begin{proof} Again, we follow \cite{BeDeGi93} to prove the first Assertion. Let $\{\lambda_\mu,\phi_\mu\}$ be a complete spectral
resolution of $(D_1)_{\BB_S^+}$. We obtain as
above that
$$
-\partial_t\beta(\phi_1,\phi_2,D_1,\BB_S^+)(t)=\sum_\mu\lambda_\mu
e^{-t\lambda_\mu}\gamma_\mu^{D_1}(\phi_1)\gamma_\mu^{D_1}(\phi_2)\,.
$$
We restrict henceforth to $\lambda_\mu>0$ since the
contribution of zero eigenvalues to the above sum is zero. Let
$$\psi_\mu:=\frac{A\phi_\mu}{\sqrt{\lambda_\mu}}.$$
Then $\{\lambda_\mu,\psi_\mu\}$ is a spectral resolution of $D_2$ on $\operatorname{Range}(A)=\ker(D_2)^\perp$ with Dirichlet boundary conditions.
Since $A\phi_\mu|_{\partial M}=0$, the boundary terms vanish and we may express:
\begin{eqnarray*}
&&\gamma_\mu^{D_2}(Af)=\int_M\langle Af,\psi_\mu\rangle\operatorname{dvol}_m
=\frac1{\sqrt{\lambda_\mu}}\int_M\langle Af,A\phi_\mu\rangle\operatorname{dvol}_m\\
&&\qquad=\frac1{\sqrt{\lambda_\mu}}\int_M\langle
f,A^*A\phi_\mu\rangle\operatorname{dvol}_m=\sqrt{\lambda_\mu}\gamma_\mu^{D_1}(f)\,.
\end{eqnarray*}
This then permits us to express
$$\beta(A\phi_1,A\phi_1,D_2,\BB^-)(t)=\sum_\mu \lambda_\mu e^{-t\lambda_\mu}\gamma_\mu^{D_1}(\phi_1)\gamma_\mu^{D_1}(\phi_2)$$
which yields the identity
$$-\partial_t\beta(\phi_1,\phi_2,D_1,\BB_S^+)(t)=\beta(A\phi_1,A\phi_2,D_2,\BB^-)(t)\,.$$
Assertion (1) now follows by equating terms in the asymptotic
expansion in exactly the same fashion as was used to establish
Assertion (1) of Lemma~\ref{lem-5.2} (the extra negative sign can not be absorbed into $D$).

We apply the relations of Equation (\ref{eqn-5.c}) and use the fact that
$e_{\ell,1}^+=c_{\ell-2,2}^-=0$ to examine
$$
\left\{\phi_1^{(1)}\phi_2b^{(\ell-2)},\ \phi_1^{(1)}\phi_2bb^{(\ell-3)},\ \phi_1^{(1)}\phi_2^{(1)}b^{(\ell-3)},\
  \phi_1^{(1)}\phi_2^{(1)}(b^2)^{(\ell-4)}\right\}.
$$
The assumption that $\ell\ge6$ is employed to ensure that
$S^2(\phi_1^{(\ell-3)}\phi_2^{(1)}+\phi_1^{(1)}\phi_2^{(\ell-3)})$ does not produce such a term. We compute at the
boundary component $x=0$:
\medbreak\quad
$e_{\ell,2}^+\phi_1^{(1)}\phi_2E_1^{(\ell-3)}
    =-\Xi_\ell \phi_1^{(1)}\phi_2b^{(\ell-2)}+2\cdot\Xi_\ell \phi_1^{(1)}\phi_2bb^{(\ell-3)}+...$,
\smallbreak\quad
$e_{\ell,3}^+\phi_1^{(1)}\phi_2^{(1)}E_1^{(\ell-4)}
   =e_{\ell,3}^+\phi_1^{(1)}\phi_2^{(1)}b^{(\ell-3)}-e_{\ell,3}^+\phi_1^{(1)}\phi_2^{(1)}(b^2)^{(\ell-4)}+...$,
\smallbreak\quad
$d_{\ell,3}^+S\phi_1^{(1)}\phi_2E_1^{(\ell-4)}=d_{\ell,3}^+\phi_1^{(1)}\phi_2bb^{(\ell-3)}+...$,
\smallbreak\quad
$-\frac2{\ell+1}c_{\ell-2,1}^-\{(\phi_1^{(1)}+b\phi_1)^{(\ell-2)}(\phi_2^{(1)}+b\phi_2)+
 (\phi_1^{(1)}+b\phi_1)(\phi_2^{(1)}+b\phi_2)^{(\ell-2)}\}$
\smallbreak\qquad\qquad
$=-\Xi_\ell \phi_1^{(1)}\phi_2b^{(\ell-2)}-\Xi_\ell(\ell-2)\phi_1^{(1)}\phi_2bb^{(\ell-3)}
-2(\ell-2)\Xi_\ell \phi_1^{(1)}\phi_2^{(1)}b^{(\ell-3)}+...$,
\smallbreak\quad
$-\frac2{\ell+1}e_{\ell-2,1}^-(\phi_1^{(1)}+b\phi_1)(\phi_2^{(1)}+b\phi_2)E_2^{(\ell-4)}$
\smallbreak\qquad\qquad
$=-\frac2{\ell+1}e_{\ell-2,1}^-\{-\phi_1^{(1)}\phi_2bb^{(\ell-3)}-\phi_1^{(1)}\phi_2^{(1)}b^{(\ell-3)}
-\phi_1^{(1)}\phi_2^{(1)}(b^2)^{(\ell-4)}\}+...$.
\medbreak\noindent This gives us the following relations:
$$\begin{array}{lrl}
(a)&\phi_1^{(1)}\phi_2b^{(\ell-2)}:&-\Xi_\ell=-\Xi_\ell,\\
(b)&\phi_1^{(1)}\phi_2bb^{(\ell-3)}:&2\cdot\Xi_\ell+d_{\ell,3}^+=-\Xi_\ell(\ell-2)+\frac2{\ell+1}e_{\ell-2,1}^-,\\
(c)&\phi_1^{(1)}\phi_2^{(1)}b^{(\ell-3)}:&e_{\ell,3}^+=-2(\ell-2)\Xi_\ell+\frac2{\ell+1}e_{\ell-2,1}^-,\\
(d)&\phi_1^{(1)}\phi_2^{(1)}(b^2)^{(\ell-4)}:&-e_{\ell,3}^+=\frac2{\ell+1}e_{\ell-2,1}^-\,.
\end{array}$$
This then yields the following 3 relations:
\medbreak (1) (c)+(d): $0=-2(\ell-2)\Xi_\ell+2\cdot\frac2{\ell+1}e_{\ell-2,1}^-$ so
$e_{\ell-2,1}^-=(\ell-2)\frac{\ell+1}2\cdot\Xi_\ell=(\ell-2)\Xi_{\ell-2}$.
\medbreak (2) (d)-(c): $-2e_{\ell,3}^+=2(\ell-2)\Xi_\ell$ so $e_{\ell,3}^+=(2-\ell)\Xi_\ell$.
\medbreak (3) (c)-(b): $-d_{\ell,3}^++e_{\ell,3}^+-2\cdot\Xi_\ell=-(\ell-2)\Xi_\ell$ so
$d_{\ell,3}^+=e_{\ell,3}^++(\ell-4)\Xi_\ell=-2\cdot\Xi_\ell$.
\end{proof}

We now work in dimension $m\ge2$ to examine
\begin{eqnarray*}
&&\beta_\ell^{\partial M}(\phi_1,\phi_2,D,\BB^-)\\
&=&\int_{\partial
M}\left
\{c_{\ell,1}^-\phi_1^{(\ell)}\phi_2+e_{\ell,1}^-\phi_1\phi_2E^{(\ell-2)}+r_{\ell,1}^-\phi_1\phi_2\rho_{mm}^{(\ell-2)}+...
\right\}\operatorname{dvol}_{m-1}\,.
\end{eqnarray*}
Let $M_1:=[0,1]$ and $\alpha\in C^\infty(M_1)$ satisfy $\alpha|_{\partial M_1}=0$.
Let
$$
D_1:=-\partial_r^2,\quad
M_2:=M_1\times S^1,\quad D_2:=D_1-e^{-2\alpha(r)}\partial_\theta^2\,.
$$
\goodbreak\begin{lemma}\label{lem-5.4}
\ \begin{enumerate}
\item If $\ell\ge2$, then
$0=\beta_\ell^{\partial M}(1,e^{\alpha(r)},-\partial_r^2,\BB^-)$.
\smallbreak\item $r_\ell^-=\frac12(\ell-2)\Xi_\ell$.
\end{enumerate}
\end{lemma}

\goodbreak\begin{proof} We follow the treatment in \cite{BeDeGi93} to prove Assertion (1). We consider the function
$u(r,t)=e^{-tD_{1,\BB^-}}1$. This solves the equations

$$(\partial_t+D_1)u=0,\quad\lim_{t\rightarrow0}u(\cdot,t)=1\text{ in }L^2(M_1),\quad
\BB^-u=0\,.$$
Since $u$ also solves the equations
$$(\partial_t+D_2)u=0,\quad\lim_{t\rightarrow0}u(\cdot,t)=1\text{
in }L^2(M_2),\quad \BB^-u=0\,,
$$
we also have that $u(\cdot,t)=e^{-tD_{2,\BB}}1$ as well. Since
$\operatorname{dvol}_{M_2}=e^\alpha drd\theta$,
\begin{eqnarray*}
\beta_{M_2}(1,e^{-\alpha},D_2,\BB^-)(t)&=&\int_{r=0}^1\int_{\theta=0}^{2\pi}u(r,t)e^{-\alpha(r)}e^{\alpha(r)}d\theta dr\\\
&=&2\pi\int_{r=0}^1u(r,t)dr=2\pi\beta_{M_1}(1,1,D_1,\BB^-)(t)\,.
\end{eqnarray*}
Since the structures are flat on $M_1$, $\beta_{\ell}^{\partial M_1}(1,1,D_1,\BB^-)=0$ for $\ell>0$ and $\Delta_{M_1}^k1=0$. We
equate terms in the asymptotic expansion to see $\beta_\ell^{\partial M_2}(1,e^{-\alpha(r)},D_2,\BB^-)=0$ for $\ell>0$ as well.

We apply Assertion (1). We use the formalism of Equation (\ref{eqn-1.e}). We have $ds_{M_2}^2=dr^2+e^{2\alpha(r)}d\theta^2$
where
$\alpha(0)=0$ and $\alpha(r)=0$ near $\alpha=1$. We compute:
$$
\begin{array}{ll}
\Gamma_{122}=\Gamma_{212}=-\Gamma_{221}=e^{2\alpha}\alpha^{(1)},&\omega_1=\frac12e^{-2\alpha}\Gamma_{221}=-\frac12\alpha^{(1)},\\
\omega_2=0,&E^{(\ell-2)}=\frac12\alpha^{(\ell)}+...,\vphantom{\vrule height 12pt}\\
\phi_1^{(\ell)}=0+...,&\phi_2^{(\ell)}=-\alpha^{(\ell)}+...,\vphantom{\vrule height 12pt}\\
\rho_{mm}^{(\ell-2)}=-\alpha^{(\ell)}+....\vphantom{\vrule height 12pt}
\end{array}$$
We examine the coefficient of $\alpha^{(\ell)}$ in $\beta_\ell$ for $\ell$ even:
\begin{eqnarray*}
&&\textstyle c_{\ell,1}^-\phi_1\phi_2^{(\ell)}=-\Xi_\ell\alpha^{(\ell)}+...,\\
&&\textstyle e_{\ell,1}^-\phi_1\phi_2E^{(\ell-2)}=\frac12\ell\cdot\Xi_\ell\alpha^{(\ell)}+....,\\
&&\textstyle r_\ell^-\phi_1\phi_2\rho_{mm}^{(\ell-2)}=-r_\ell^-\alpha^{(\ell)}+...\,.
\end{eqnarray*}
It now follows from Assertion (1) that $r_\ell^-=\frac12(\ell-2)\Xi_\ell$. This completes the proof of
Lemma~\ref{lem-5.4} and thereby completes the proof of Theorem~\ref{thm-1.9} as well.
\end{proof}

\section{Estimating the heat trace asymptotics on a closed manifold}\label{sect-6}

In this section, we shall prove Theorem~\ref{thm-1.2}. 
We shall proceed purely formally and shall use the discussion in Sections
1.7-1.8 of
\cite{G94} (which is based on the Seeley calculus \cite{Se68b,Se69a})  to justify our formal procedures.
As in Equation (\ref{eqn-1.b}), let
$$
D=
-g^{ij}\partial_{x_i}\partial_{x_j}-A^k\partial_{x_k}-B
$$
be an operator of Laplace type. Throughout this section,
$C=C(M,g,D)$ will denote a generic constant which depends only on
$(M,g,D)$ (and hence also implicitly on $m$) but not on
$n$; $c(m)$ will denote a generic constant which only depends on $m$. If we take $D=\Delta_g$, then $C=C(M,g)$.

We introduce coordinates $\xi=(\xi_1,...,\xi_m)$ on the cotangent bundle to express a covector
in the form $\xi=\xi_idx^i$. The symbol of $D$ is
$p_2(x,\xi)+p_1(x,\xi)+p_0(x)$ where:
$$p_2(x,\xi):=g^{ij}(x)\xi_i\xi_j,\quad p_1(x,\xi):=A^k(x)\xi_k,\quad\text{and}\quad p_0=B\,.$$
There are suitable normalizing constants involving factors of $\sqrt{-1}$ which we ignore in the interests of simplicity
 henceforth since they
play no role in the estimates we shall be deriving. Let $\mathcal{C}:=\mathbb{C}-[0,\infty)$ be the slit complex plane and let
$\lambda\in\mathcal{C}$. Following the discussion in Lemma 1.7.2 of \cite{G94}, one defines inductively:
\begin{equation}\label{eqn-6.a}
\begin{array}{l}
r_0(x,\xi,\lambda):=(|\xi|^2-\lambda)^{-1},\\
r_n(x,\xi,\lambda):=-r_0(x,\xi,\lambda)\cdot
\displaystyle\sum_{|\alpha|+j+2-k=n,j<n}d_\xi^\alpha p_k(x,\xi)\cdot d_x^\alpha r_j(x,\xi,\lambda)/\alpha!\,.
\end{array}\end{equation}
In this sum $k=0,1,2$ and $|\alpha|\le2-k$. The symbol of
$e^{-tD}$ is given by:
$$e_0(x,\xi,t)+...+e_n(x,\xi,t)+...$$ where, following Equation (1.8.4) of \cite{G94}, one sets:
$$
e_n(x,\xi,t):=\frac1{2\pi\sqrt{-1}}\int_\gamma e^{-t\lambda}{r_n}(x,\xi,\lambda)d\lambda\,;
$$
here $\gamma$ is a suitable contour about the positive real axis in the complex plane. Then, following Equation (1.8.3) of
\cite{G94}, one may obtain the local heat trace invariants of Equation (\ref{eqn-1.c}) by setting:
\begin{equation}\label{eqn-6.b}
a_n(x,D)=\left(\sqrt{\det(g_{ij})}\right)^{-1}\int_{\mathbb{R}^m} e_n(x,\xi,1)d\xi\,.
\end{equation}

To measure the degree of an expression in the derivatives of the symbol, we set:
$$
\operatorname{degree}(d_x^\alpha g^{ij})=|\alpha|,\quad
\operatorname{degree}(d_x^\alpha A^k)=|\alpha|+1,\quad
\operatorname{degree}(d_x^\alpha B)=|\alpha|+2\,.
$$
Note that if $D$ is the scalar Laplacian, then $B=0$ and $A^k=g^{-1}\partial_{x_i}g^{ij}g$ has degree 1 in the
derivatives of the metric so this present definition is consistent with our previous definition in this special case. It
is immediate from the definition that
$r_0$ is of total degree
$0$ in the jets of the symbol of
$D$. Furthermore, since 
$$\operatorname{degree}(d_\xi^\alpha p_k)=2-k\quad\text{and}\quad\operatorname{degree}(d_x^\alpha
r_j)=|\alpha|+\operatorname{degree}(r_j)\,,$$
 we have by induction that
\begin{equation}\label{eqn-6.c}
\operatorname{degree}(r_n)=n\,.
\end{equation}
 There is a similar grading on the variables $(\xi,\lambda)$. One defines:
$$
\operatorname{weight}(\xi_i)=1\quad\text{and}\quad
\operatorname{weight}(\lambda)=2\,.
$$
It is then immediate that $r_0$ has weight $-2$ in $(\xi,\lambda)$. 
Clearly 
$$\operatorname{weight}(d_\xi^\alpha p_k)=k-|\alpha|\quad\text{and}\quad\operatorname{weight}(d_x^\alpha r_j)=\operatorname{weight}(r_j)\,.$$
Thus it then also follows
by induction from Equation (\ref{eqn-6.a}) that
\begin{equation}\label{eqn-6.d}
\operatorname{weight}(r_n)=-2-n\,.
\end{equation}
Let $n$ be odd. Since  the weight of $r_n(x,\xi,\lambda)$ is $-n-2$ in $(\xi,\lambda)$, it follows that $e_n(x,\xi,1)$ is an odd
function of $\xi$ and hence the integral in Equation (\ref{eqn-6.b}) vanishes in this instance. This yields $a_n(x,D)=0$ for
$n$ odd. Let $[\cdot]$ be the greatest integer function.

\goodbreak\goodbreak\begin{lemma}\label{lem-6.1}
\  
\begin{enumerate}
\item We may expand $r_n$ in the form:
$$\displaystyle r_n(x,\xi,\lambda)=\sum_{j=[\frac12n]+1}^{2n+1}\ \ \sum_{|\beta|=2j-n-2} q_{n,m,j,\beta}(x,g)\xi^\beta
r_0^j(x,\xi,\lambda)\,.$$
\item There exists a constant $C(M,g)$ so that if $n=2\bar n>0$ and if $|\beta|=2j-n-2$, then
$$\displaystyle\left|\int_{\mathbb{R}^m}\int_\gamma e^{-\lambda}r_0^j(x,\xi,\lambda)\xi^\beta d\lambda
d\xi\right|\le \frac{C(M,g)}{\bar n!}^n\,.$$
\end{enumerate}
\end{lemma}

\goodbreak\begin{proof}
 We apply the recursive scheme of Equation (\ref{eqn-6.a}) to obtain an expression for $r_n$ of the form given in Assertion (1).
By Equation (\ref{eqn-6.c}), $r_n$ has degree $n$ in the derivatives of the symbol of $D$. Thus there are at most $n$
$x$-derivatives of
$r_0$ which are involved in the process.  Each $x$-derivative of $r_0$ adds one power of
$r_0$ (other variables can be differentiated as well of course so we are obtaining an upper bound not a sharp estimate).  Each
step in the induction process adds 1 power of
$r_0$. Thus
$j\le2n+1$. By Equation (\ref {eqn-6.d}),
$r_n$ is homogeneous of weight $-n-2$ in $(\xi,\lambda)$. Since
$|\beta|-2j=-n-2$ and $|\beta|\ge0$, we may conclude that $j\ge 1+\frac12n\ge[\frac12n]+1$. Assertion (1) now follows.

We use the Cauchy integral formula to estimate:
$$
\left|\int_{\mathbb{R}^m}\int_\gamma e^{-\lambda}(|\xi|^2-\lambda)^{-j}\xi^\beta d\lambda d\xi\right|
\le\frac1{(j-1)!}\left|\int_{\mathbb{R}^m}e^{-|\xi|^2}\xi^\beta d\xi\right|\,.$$
The quadratic form $g^{ij}$ is positive definite. Thus we may estimate $|\xi|^2\ge\varepsilon|\xi|_e^2$ for
some $\varepsilon=\varepsilon(M,g)>0$ where $|\xi|_e^2=\xi_1^2+...+\xi_m^2$ is the usual Euclidean length. Note that
$|\xi^\beta|\le|\xi|_e^{|\beta|}$. Since
$e^{-|\xi|^2}\le e^{-\varepsilon|\xi|_e^2}$, we may use spherical coordinates to estimate:
$$\left|\int_{\mathbb{R}^m}\int_\gamma e^{-\lambda}(|\xi|^2-\lambda)^{-j}\xi^\beta d\lambda
d\xi\right|\le\frac1{(j-1)!}\int_{r=0}^\infty e^{-\varepsilon
r^2}r^{|\beta|+m}dr\operatorname{vol}_{m-1}(S^{m-1},g_{S^{m-1}})\,.$$ Since $|\beta|\le 2j\le 4n+4$ is uniformly and linearly bounded in $n$,
we may rescale to remove $\varepsilon$ in $e^{-\varepsilon r^2}$ at the cost of introducing a suitable multiplicative
constant. We may then evaluate the integral to estimate:
$$
 \left|\int_{\mathbb{R}^m}\int_\gamma e^{-\lambda}(|\xi|^2-\lambda)^{-j}\xi^\beta d\lambda d\xi\right|\le
C(M,g)^n\frac{(\frac{|\beta|+m}2)!}{(j-1)!}\,.
$$ 
Since $j-1-\frac12|\beta|=\bar n$ the desired estimate follows; the shift by
$m$ can be absorbed into $C(M,g)^n$ since we have restricted to $n>0$.
\end{proof}

Let $D_\varepsilon^{\mathbb{C}}\subset\mathbb{C}^m$ be the complex polydisk
of radius
$\varepsilon$ of real dimension
$2m$ about the origin in
$\mathbb{C}^m$ given by setting:
$$
D_\varepsilon^{\mathbb{C}}:=\left\{\vec z=(z_1,...,z_m)\in\mathbb{C}^m:|z_i|\le\varepsilon\quad\text{for}\quad 1\le i\le
m\right\}\,.
$$
We  let $D_\varepsilon^{\mathbb{R}}=D_\varepsilon^{\mathbb{C}}\cap\mathbb{R}^m$ be the corresponding real polydisk. We
also consider the submanifold $S_\varepsilon$ of real dimension
$m$ in $\mathbb{C}^m$ (which is not the boundary either of the complex polydisk $D_\varepsilon^{\mathbb{C}}$ or of the real polydisk
$D_\varepsilon^{\mathbb{R}}$) given by:
$$S_\varepsilon:=\{\vec z\in\mathbb{C}^m:|z_i|=\varepsilon\quad\text{for}\quad 1\le i\le m\}\,.$$
We consider the holomorphic $m$-form
$$dw=(2\pi\sqrt{-1})^{-m}dw_1...dw_m\,.$$
Let $f$ be a holomorphic function on the interior of $D_\varepsilon^{\mathbb{C}}$ which extends continuously to all of
$D_\varepsilon^{\mathbb{C}}$ and let $\alpha$ is a multi-index. If
$z$ belongs to the interior of the polydisk $D_\varepsilon^{\mathbb{C}}$, then we shall define:
$$
\mathcal{I}_\alpha(f)(z):=\int_{w\in S_\varepsilon}f(w)(w_1-z_1)^{-1-\alpha_1}...(w_m-z_m)^{-1-\alpha_m}dw\,.
$$
We may then use the Cauchy integral formula to represent:
$$
\partial_z^\alpha f(z)=\alpha!\mathcal{I}_\alpha(f)\quad\text{for}\quad
z\in\operatorname{int}(D_\varepsilon^{\mathbb{C}})\,.
$$
Let $\beta=\beta(i,\alpha)$ be the multi-index $(\alpha_1,...,\alpha_{i-1},\alpha_i+1,\alpha_{i+1},...,\alpha_m)$. We
then have:
\begin{equation}\label{eqn-6.e}
\partial_{x_i}\mathcal{I}_\alpha(f)(x)=(\alpha_i+1)\cdot\mathcal{I}_{\beta}(f)(x)\,.
\end{equation}

We introduce variables $\{f_\nu\}$ for the $\{g^{ij},A^k,B\}$ variables; we have a total of $\frac12m(m-1)+m+1$ such
variables. Since we are in the real analytic setting, we can choose real analytic coordinates about each point $P$ of $M$
which are real analytically equivalent to the polydisk
$D_2^{\mathbb{R}}(P)$ of radius 2 in such a way that the variables $\{f_\nu\}$ extend continuously to $D_2^{\mathbb{C}}(P)$
with $f_\nu$ holomorphic on the interior of $D_2^{\mathbb{C}}(P)$. The functions $|f_\nu|$ are uniformly bounded on
$D_2^{\mathbb{C}}(P)$. If
$z\in D_1^{\mathbb{R}}(P)$ and
$|w|\in S_2^{\mathbb{C}}(P)$, then $|z_i-w_i|\ge1$ and thus we have uniform estimates
\begin{equation}\label{eqn-6.f}
|\mathcal{I}_\alpha(f_\nu)(z)|\le C(M,D)\quad\text{for any}\quad\nu,\alpha\,.
\end{equation}

 We decompose
$r_n$ in terms of monomials of the form
\begin{equation}\label{eqn-6.g}
r_0^j\xi^\beta\cdot g^{i_1j_1}\cdot...\cdot g^{i_aj_a}\cdot I_{\alpha_1}(f_{\nu_1})\cdot
...\cdot I_{\alpha_b}(f_{\nu_b})\,.
\end{equation}
Here we assume $\operatorname{degree}\{\partial_\alpha^xf_{\nu_i}\}>0$ since we have made explicit the dependence on the
variables of degree 0. Thus
$b\le n$ since, by Equation (\ref{eqn-6.c}),
$r_n$ is homogeneous of degree
$n$ in the jets of the symbol. There are no $g^{ij}$ variables in $r_0$. Each multiplication by $\partial_\xi^\alpha p_2$
can add at most one $g^{ij}$ variable; each multiplication by $\partial_{\xi_i}^\alpha p_1$ or $p_0$ adds no $g^{ij}$
variable. Each application of $\partial_x^\alpha$ to $ r_j$ does not add a $g^{ij}$ variable (and can in fact reduce the
number of $g^{ij}$ variables if they are differentiated). Thus the number of $g^{ij}$ variables is at most $n$. Thus in
considering monomials of the form given in Equation (\ref{eqn-6.g}), we may assume $a\le n$. We summarize these constraints:
\begin{equation}\label{eqn-6.h}
j\le 2n+1,\quad -n-2=|\beta|-2j,\quad a\le n,\quad\text{and}\quad b\le n\,.
\end{equation}

\goodbreak\begin{lemma} Let $c(m):=50m^2$. We can decompose $r_n$ as the sum of at most $c(m)^n n!$ monomials of the form given
in  {\rm Equation (\ref{eqn-6.g})} satisfying the constraints of {\rm Equation (\ref{eqn-6.h})} where the coefficient of each
monomial has absolute value at most 1.
\end{lemma}

\goodbreak\begin{proof} Since $r_0$ can be written as a single monomial with coefficient $1$, we proceed by induction.
\begin{enumerate}
\item Consider $-r_0\partial_{\xi_k}p_2\cdot\partial_{x_k}r_{n-1}$. Each $k$ generates $m$ terms so there are $m^2$ terms
generated in this way. Differentiating $r_0^j$ generates at most $3n$ terms since $j\le 3n$ by Equation (\ref{eqn-6.h}).
Differentiating the $g^{ij}$ variables generates at most $n$ terms since $a\le n$. Differentiating the $\mathcal{I}$
variables generates at most $b+\sum|\alpha_i|\le2n$ terms by Equation (\ref{eqn-6.e}). Thus we generate at most
$m^2(3n+n+2n)=6m^2n$ terms from each monomial of $r_{n-1}$. This can be written in terms of at most
\medbreak $6m^2 n\cdot c(m)^{m-1}(n-1)!=6m^2c(m)^{m-1}n!$ monomials.
\medbreak\item Consider $-r_0\partial_{\xi_{k_1}}\partial_{\xi_{k_2}}p_2\cdot\partial_{x_{k_1}}\partial_{x_{k_2}}r_{n-2}$. A
similar argument shows this generates at most $m^2(6n)(6(n-1))$ new terms from each monomial of $r_{n-2}$. This can be
written in terms of at most
\medbreak $36m^2n(n-1)\cdot c(m)^{n-2}(n-2)!\le36m^2\cdot c(m)^{n-1}n!$ monomials.
\medbreak\item Consider $-r_0A^k\xi_k r_{n-1}$. This can be written in terms of at most
\medbreak$m\cdot c(m)^{m-1}(n-1)!\le m^2c(m)^{m-1}n!$ monomials.
\medbreak\item Consider $-r_0A^k\partial_{x_k}r_{n-2}$. This can be written in terms of at most
\medbreak$6mn\cdot c(m)^{n-2}(n-2)!\le 6m^2c(m)^{n-1}n!$ monomials.
\medbreak\item Consider $-r_0Br_{n-2}$. This can be written in terms of at most
$$c(m)^{n-2}(n-2)!\le m^2c(m)^{n-1}n!\quad\text{terms}.$$
\end{enumerate}
The above argument shows that $r_n$ can be decomposed as the sum of at most of $50m^2\cdot c(m)^{n-1}n!=c(m)^n\cdot
n!$ monomials each of which has a coefficient of absolute value at most $1$.
\end{proof}

\goodbreak\begin{proof}[Proof of Theorem~\ref{thm-1.2} (1)] We consider monomials where the coefficient has absolute value at most $1$.
We have shown that there exists a constant
$c(m)$ so that $r_n$ can be written in terms of at most $c(m)^nn!$ such monomials. We may then use
the constraints of Equation (\ref{eqn-6.h}), the estimates of Equation (\ref{eqn-6.f}), and the estimate of
Lemma~\ref{lem-6.1} to construct a new constant $\tilde C(M,g)$ and complete the proof of Theorem~\ref{thm-1.2} (1) by
bounding:
\medbreak
\qquad$|a_n(x,D)|\le c(m)^nn!\cdot C(M,g,D)^{2n}\cdot C(M,g)^n\frac1{\bar n!}\le\tilde C(M,g,D)^n\bar n!$.
\end{proof}

\goodbreak\begin{proof}[Proof of Theorem~\ref{thm-1.2} (2)] Let $P$ be a point of a closed real analytic
Riemannian manifold $(M,g)$. Let $f$ be a real analytic function on $M$ so that $df(P)\ne0$. Since
$f$ is continuous and $M$ is compact, $|f|$ is bounded. By rescaling and shifting $f$, we may suppose without loss of
generality that $f(P)=0$ and that $|f(x)|\le 1$ for all points $x$ of $M$. We make a real analytic change of
coordinates to assume that $g^{ij}(P)=\delta_{ij}$ and that $f(x)=c_f\cdot x_1$ near $P$. We shall choose $\varepsilon_k=\pm1$ recursively and define:
$$h(x)=\sum_{k=3}^\infty\varepsilon_k2^{-k}f(x)^{2k}\,.$$
This series converges uniformly in the real analytic topology so $h$ is real analytic. Let
$\mathcal{E}_{\bar n}(\cdot)$ be a generic invariant which only depends on the parameters indicated. Let $g_h=e^{2h}g$. 
Let $\bar n\ge3$. We use Theorem
\ref{thm-1.8} to see that:
\medbreak\qquad
$(\partial_{x_1}^{2\bar n}h)(P)=\varepsilon_{\bar n}2^{-\bar n}c_f^{2\bar n}(2\bar n)!
+\mathcal{E}_{\bar n}^1(\varepsilon_1,...,\varepsilon_{\bar n-1})$,
\medbreak\qquad
$\tau_{g_h}(P)=c_m(\partial_{x_1}^2h)(P)+\text{lower order terms}\quad\text{for some}\quad |c_m|\ge1$,
\medbreak\qquad
$(-1)^{\bar n-1}\Delta_{g_h}^{\bar n-1}\tau_{g_n}(P)=
\varepsilon_{\bar n}c_m2^{-\bar n}c_f^{2\bar n}(2\bar n)!
+\mathcal{E}_{\bar n}^2(\varepsilon_1,...,\varepsilon_{\bar n-1},g)$,
\medbreak\qquad
$a_{2\bar n}(P,\Delta_g)=(-1)^{\bar n-1}\frac{\bar n\cdot \bar n!}{(2\bar n+1)!}\Delta^{\bar n-1}\tau+\text{lower order
terms}$
\medbreak\qquad\qquad\qquad\quad
$=c_m\frac{\bar n\cdot \bar n!}{(2\bar n+1)!}c_f^{2\bar n}\varepsilon_{\bar n}2^{-\bar n}(2\bar
n)!+\mathcal{E}_{\bar n}^3(\varepsilon_1,...,\varepsilon_{\bar n-1},g)$.
\medbreak\noindent
We set
$$\varepsilon_{\bar n}:=\left\{\begin{array}{lll}
+1&\text{if}&c_m\mathcal{E}_{\bar n}^3(\varepsilon_1,...,\varepsilon_{\bar n-1},g)\ge0\\
-1&\text{if}&c_m\mathcal{E}_{\bar n}^3(\varepsilon_1,...,\varepsilon_{\bar n-1},g)<0\vphantom{\vrule height 12pt}
\end{array}\right\}\,.$$
With this choice of $\varepsilon_{\bar n}$, there is no cancellation. As $\frac12\frac{\bar n}{2\bar
n+1}\ge\frac3{14}$ for $\bar n\ge3$, we obtain the desired estimate:
\medbreak\hfill
$\left|a_{2\bar n}(P,\Delta_g)\right|\ge c_m\frac{\bar n\cdot \bar n!}{(2\bar n+1)!}c_f^{2\bar n}2^{-\bar n}(2\bar n)!
\ge\frac {\bar n}{2\bar n+1}c_f^{2\bar n}2^{-\bar n}\cdot \bar n!\ge\left(\frac3{14}c_f^2\right)^{\bar n}\bar n!$.
\hfill\vphantom{.}
\end{proof}

\section{Growth of heat content asymptotics}\label{sect-7}

This section is devoted to the proof of Theorem~\ref{thm-1.4}. We first examine a product manifold
$[0,1]\times N$. Let $\{\varepsilon_{\bar\ell}\}$ be a sequence of signs to be chosen recursively. We replace the
function $f(x)$ of the previous section by $\sin(x)$ and define:
$$
h(x):=\sum_{\nu=1}^\infty\varepsilon_\nu 2^{-\nu}\sin(x)^{2\nu}\,.
$$
This series converges in the real analytic topology to a real analytic function $h$ which is periodic with period $2\pi$
and which satisfies $h(0)=h(2\pi)=0$. We set
$$g_M:=e^{2h}(dx^2+g_N)\,.$$
The inward unit normal is given at $0$ by $\nu(0)=\partial_x$ and at $2\pi$ by $\nu(2\pi)=-\partial_x$. 
If $j$ is odd, then $\{\partial_x^jh\}(0)=\{\partial_x^jh\}(2\pi)=0$ since $h$ is an even function.
 And clearly we have that
$\{(\partial_x^j)h\}(0)=\{(-\partial_x)^jh\}(2\pi)$ if $j$ is even. Consequently
$$
h^{(j)}(0)=h^{(j)}(2\pi)\quad\text{for any}\quad j\,.
$$
This ensures that the behaviour of $h$ is the same on the boundary components and gives rise to the factor of
$2\operatorname{vol}_{m-1}(N,g_N)$ in Equation (\ref{eqn-7.a}) below. We have:
$$
h^{(2\bar\ell)}(0)=\varepsilon_{\bar\ell}\cdot
2^{-\bar\ell}(2\bar\ell)!+\mathcal{E}_{\bar\ell}^4(\varepsilon_1,...,\varepsilon_{\bar\ell-1})\,.
$$
Since $m\ge2$, there is a non-zero constant $c_m$ with $|c_m|\ge1$ which only depends on $m$ and not on $\bar\ell$ so
that:
$$\rho_{mm}^{(2\bar\ell-2)}(0)=\varepsilon_{\bar\ell}\cdot c_m
2^{-\bar\ell}(2\bar\ell)!+\mathcal{E}_{\bar\ell}^5(\varepsilon_1,...,\varepsilon_{\bar\ell-1},g_N)\,.
$$
We may then apply Theorem~\ref{thm-1.9} to express:
\begin{equation}\label{eqn-7.a}
\begin{array}{l}
\beta_{2\bar\ell}^{\partial M}(1,1,\Delta_{M,g_M},\BB^-)
=\textstyle\varepsilon_\ell\left\{\frac12(2\bar\ell-2)\Xi_{2\bar\ell}
c_m2^{-\bar\ell}(2\bar\ell)!\cdot2\operatorname{vol}_{m-1}(N,g_N)\right\}\\
\qquad\qquad\qquad\qquad\qquad\quad+\mathcal{E}_{2\bar\ell}^6(\varepsilon_1,...,\varepsilon_{\bar\ell-1},g_N)\,.
\vphantom{\vrule height 13pt}\end{array}\end{equation}
Set
$$
\varepsilon_{\bar\ell}:=\left\{\begin{array}{lll}
+1&\text{if}&\mathcal{E}_{2\bar\ell}^6(\varepsilon_1,...,\varepsilon_{\bar\ell-1},g_N)>0\\
-1&\text{if}&\mathcal{E}_{2\bar\ell}^6(\varepsilon_1,...,\varepsilon_{\bar\ell-1},g_N)\le0\vphantom{\vrule height 12pt}
\end{array}\right\}\,.$$
Since there is no cancellation in Equation (\ref{eqn-7.a}), we may estimate:
$$
\left|\beta_{2\bar\ell}^{\partial M}(1,1,\Delta_{M,g_M},\BB^-)\right|\ge\frac12(2\bar\ell-2)\Xi_{2\bar\ell}
c_m(2\bar\ell)!\varepsilon_{\bar\ell}2^{-\bar\ell}\cdot2\operatorname{vol}_{m-1}(N,g_N)\,.
$$
The desired estimate in Assertion (1) of Theorem~\ref{thm-1.4} now follows since:\medbreak\qquad
$\displaystyle\phantom{\ge}\left|\frac12(2\bar\ell-2)\Xi_{2\bar\ell}
c_m(2\bar\ell)!\varepsilon_{\bar\ell}2^{-\bar\ell}\right|
\ge(2\bar\ell-2)\frac{2}{2\bar\ell+1}...\frac2{3}\frac2{\sqrt\pi}2^{-\bar\ell}1\cdot 2\cdot 3...\cdot 2\bar\ell$
\medbreak\qquad
$\displaystyle=\frac{2\bar\ell-2}{2\bar\ell+1}2\cdot 4\cdot...\cdot 2\bar\ell
\ge\frac4{14} 2^{\bar\ell}\bar\ell!\ge\bar\ell!$\quad for\quad $\bar\ell\ge3$.
\medbreak We now turn to the case of the ball and apply a similar analysis to establish Assertion (2) of
Theorem~\ref{thm-1.4}. The functions $\sin(x)$ is now replaced by the function $(x_1^2+...+x_m^2-1)^{2\nu}$,
the operator $\partial_x$ is replaced by the radial derivative $\partial_r$, and the boundary components $x=0$ and
$x=2\pi$ are replaced by the single boundary component $r=1$. The remainder of the argument is the same and is therefore
omitted.
\hfill\qedbox

\section*{Acknowledgments} Research of M. van den Berg partially supported by the London Mathematical Society Grant
41028. Research of P. Gilkey partially supported by Project
MTM2009-07756 (Spain), by DFG
PI 158/4-6 (Germany), and by Project 174012 (Serbia). Research of K. Kirsten supported by National
Science Foundation Grant PHY-0757791. It is a pleasure to
acknowledge useful conversations with Professor Michael Berry
(University of Bristol) concerning these
matters.

\end{document}